\documentclass[12pt,a4paper,dvipsnames]{amsart}

\usepackage{amsmath, nicefrac, amsthm, verbatim, amsfonts, amssymb, xcolor, bbm}
\usepackage{graphics, xspace, enumerate}
\usepackage[a4paper,margin=3cm]{geometry}
\usepackage{multirow}

\usepackage[bb=boondox]{mathalfa}

\usepackage{graphicx}
\usepackage[colorlinks=true,citecolor=green,urlcolor=green,linkcolor=green,bookmarksopen=true,unicode=true,pdffitwindow=true]{hyperref}
\usepackage[english]{babel}
\usepackage[languagenames,fixlanguage]{babelbib}

\usepackage[labelformat=simple]{subcaption}

\usepackage{tikz}

\makeatletter
\@namedef{subjclassname@2020}{\textup{2020} Mathematics Subject Classification}
\makeatother

\theoremstyle{plain}
\newtheorem{theorem}{Theorem}[section]

\theoremstyle{definition}

\newtheorem{claim}[theorem]{Claim}
\newtheorem{question}[theorem]{Question}
\newtheorem{conjecture}[theorem]{Conjecture}

\newcommand{\bbN}{\mathbb{N}}

\newcommand{\bbP}{\mathbb{P}}
\newcommand{\bbE}{\mathbb{E}}

\newcommand{\calM}{\mathcal{M}}

\newcommand{\aps}[1]{\vert #1 \vert}

\newcommand{\OBL}[1]{\left( #1 \right)}

\newcommand{\FJADEF}[3]{{#1}:{#2}\to{#3}}

\numberwithin{equation}{section}

\definecolor{Maroon}{RGB}{140,10,0}

\hypersetup{
	colorlinks,
	linkcolor={Maroon},
	citecolor={Maroon},
	urlcolor={Maroon}
}

\title{Packing density of combinatorial settlement planning models}

\author[M.\ Puljiz]{Mate\ Puljiz}
\address[Mate Puljiz]{Department of Applied Mathematics\\
	Faculty of Electrical Engineering and Computing\\
	University of Zagreb\\ 
 Zagreb\\ 
	Croatia}
\email{mate.puljiz@fer.hr}

\author[S.\ \v{S}ebek]{Stjepan\ \v{S}ebek}
\address[Stjepan\ \v{S}ebek]{Department of Applied Mathematics\\
	Faculty of Electrical Engineering and Computing\\
	University of Zagreb\\ 
 Zagreb\\ 
	Croatia}
\email{stjepan.sebek@fer.hr}

\author[J.\ \v{Z}ubrini\'{c}]{Josip\ \v{Z}ubrini\'{c}}
\address[Josip\ \v{Z}ubrini\'{c}]{Department of Applied Mathematics\\
	Faculty of Electrical Engineering and Computing\\
	University of Zagreb\\ 
 Zagreb\\ 
	Croatia}
\email{josip.zubrinic@fer.hr}



\begin{document}
\allowdisplaybreaks[4]

\begin{abstract}
    Recently, a combinatorial settlement planning model was introduced in \cite{PSZ-21}. The idea underlying the model is that the houses are randomly being built on a rectangular tract of land according to the specified rule until the maximal configuration is reached, that is, no further houses can be built while still following that rule. Once the building of the settlement is done, the main question is what percentage of the tract of land on which the settlement was built has been used, i.e.\ what is the building density of the maximal configuration that was reached. In this article, with the aid of simulations, we find an estimate for the average building density of maximal configurations and we study what happens with this average when the size of a tract of land grows to infinity.
\end{abstract}

\maketitle

%
%
%
%

\section{Introduction}

The following problem was recently introduced in \cite{PSZ-21}: a rectangular $m\times n$ tract of land, whose sides are oriented north-south and east-west as in Figure \ref{fig:tract_of_land}, consists of $mn$ square lots of size $1 \times 1$. Each $1 \times 1$ square lot can be either empty, or occupied by a single house. A house is said to be \emph{blocked from sunlight} if the three lots immediately to its east, west and south are all occupied (it is assumed that sunlight always comes from the south\footnote{Our Southern Hemisphere friends are welcome to turn the page upside down when inspecting the figures in our paper.}). Along the boundary of the rectangular $m \times n$ grid, there are no obstructions to sunlight. We refer to the models of such rectangular tracts of land, with certain lots occupied, as configurations. Of interest are maximal configurations, where no house is blocked from the sunlight, and any further addition of a house to the configuration on any empty lot would result in either that house being blocked from the sunlight, or it would cut off sunlight from some previously built house, or both.

\begin{figure}[h]
	\centering
	\begin{tikzpicture}[scale = 0.5]
	\draw[step=1cm,black,very thin] (0, 0) grid (7,5);
	\draw [->,>=stealth] (9,1) -- (9,4);
	\node[anchor=west] at (9,4) {North};
	\end{tikzpicture}
	\caption{An example of a tract of land ($m = 5$, $n = 7$).}
	\label{fig:tract_of_land}
\end{figure}
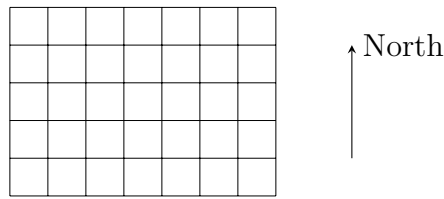

\medskip

We can encode any fixed configuration as a 0-1, $m\times n$ matrix $C$, with $C_{i,j}=1$ if and only if a house is built on the lot $(i,j)$ ($i$-th row and $j$-th column, counted from the top left corner). We can, equivalently, think of $C$ as a subset of $[m]\times[n] = \{(i,j) : 1\le i \le m,\, 1\le j\le n\}$, where, again, $(i,j)\in C$ if and only if a house is built on the lot $(i,j)$.

It is natural to define the \emph{building density} of a configuration $C$ as $\dfrac{|C|}{mn}$, where 
$$|C| = \sum_{i=1}^m\sum_{j=1}^n C_{i,j}$$
is the total number of occupied lots in the configuration $C$, i.e.\ the cardinality of $C$ when $C$ is interpreted as a subset of $[m]\times [n]$. We also call $|C|$ the \emph{occupancy} of $C$.

A configuration $C$ is said to be \emph{permissible} if no house in it is blocked from the sunlight, otherwise it is called \emph{impermissible}.

A configuration $C$ is said to be \emph{maximal} if it is permissible and no other permissible configuration strictly contains it, i.e.\ no further houses can be added to it, whilst ensuring that all the houses still get some sunlight. See Figure \ref{fig:examples} for examples of impermissible, permissible and maximal configuration on a $5 \times 4$ tract of land. Shaded squares represent houses and unshaded squares represent empty lots on the tract of land. The houses that are blocked from the sunlight are marked with the letter x.

\begin{figure}
	\begin{subfigure}{0.3\textwidth}\centering
		\begin{tikzpicture}[scale = 0.5]
		\draw[step=1cm,black,very thin] (0, 0) grid (4,5);
		\fill[blue!40!white] (0,1) rectangle (1,2);
		\fill[blue!40!white] (0,3) rectangle (1,4);
		\fill[blue!40!white] (0,4) rectangle (1,5);
		\fill[blue!40!white] (1,1) rectangle (2,2);
		\fill[blue!40!white] (1,2) rectangle (2,3);
		\fill[blue!40!white] (1,3) rectangle (2,4);
		\node[] at (1.5,3.5) {x};
		\fill[blue!40!white] (2,0) rectangle (3,1);
		\fill[blue!40!white] (2,1) rectangle (3,2);
		\fill[blue!40!white] (2,2) rectangle (3,3);
		\node[] at (2.5,2.5) {x};
		\fill[blue!40!white] (2,3) rectangle (3,4);
		\fill[blue!40!white] (3,0) rectangle (4,1);
		\fill[blue!40!white] (3,2) rectangle (4,3);
		\draw[step=1cm,black,very thin] (0, 0) grid (4,5);
		\end{tikzpicture}
		\caption{Impermissible}
	\end{subfigure}
	\begin{subfigure}{0.3\textwidth}\centering
		\begin{tikzpicture}[scale = 0.5]
		\draw[step=1cm,black,very thin] (0, 0) grid (4,5);
		\fill[blue!40!white] (0,1) rectangle (1,2);
		\fill[blue!40!white] (0,3) rectangle (1,4);
		\fill[blue!40!white] (0,4) rectangle (1,5);
		\fill[blue!40!white] (1,1) rectangle (2,2);
		\fill[blue!40!white] (1,3) rectangle (2,4);
		\fill[blue!40!white] (2,0) rectangle (3,1);
		\fill[blue!40!white] (2,1) rectangle (3,2);
		\fill[blue!40!white] (2,2) rectangle (3,3);
		\fill[blue!40!white] (2,3) rectangle (3,4);
		\fill[blue!40!white] (3,0) rectangle (4,1);
		\fill[blue!40!white] (3,2) rectangle (4,3);
		\draw[step=1cm,black,very thin] (0, 0) grid (4,5);
		\end{tikzpicture}
		\caption{Permissible}
	\end{subfigure}
	\begin{subfigure}{0.3\textwidth}\centering
		\begin{tikzpicture}[scale = 0.5]
		\draw[step=1cm,black,very thin] (0, 0) grid (4,5);
		\fill[blue!40!white] (0,0) rectangle (1,1);
		\fill[blue!40!white] (0,1) rectangle (1,2);
		\fill[blue!40!white] (0,2) rectangle (1,3);
		\fill[blue!40!white] (0,3) rectangle (1,4);
		\fill[blue!40!white] (0,4) rectangle (1,5);
		\fill[blue!40!white] (1,1) rectangle (2,2);
		\fill[blue!40!white] (1,3) rectangle (2,4);
		\fill[blue!40!white] (2,0) rectangle (3,1);
		\fill[blue!40!white] (2,1) rectangle (3,2);
		\fill[blue!40!white] (2,2) rectangle (3,3);
		\fill[blue!40!white] (2,3) rectangle (3,4);
		\fill[blue!40!white] (2,4) rectangle (3,5);
		\fill[blue!40!white] (3,0) rectangle (4,1);
		\fill[blue!40!white] (3,2) rectangle (4,3);
		\fill[blue!40!white] (3,4) rectangle (4,5);
		\draw[step=1cm,black,very thin] (0, 0) grid (4,5);
		\end{tikzpicture}
		\caption{Maximal}
	\end{subfigure}
	\caption{Examples of impermissible, permissible and maximal configuration on a $5 \times 4$ tract of land.}\label{fig:examples}
\end{figure}
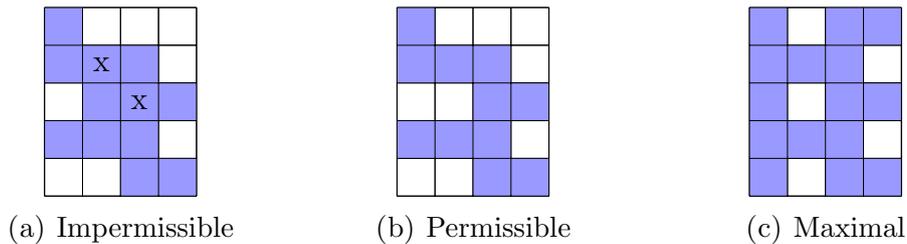

There are many natural questions related to maximal configurations. The authors in \cite{PSZ-21} were interested in those maximal configurations that achieve the highest and the lowest building density. We will focus our attention to the question of the average building density of a maximal configuration.

We were introduced to this problem by Juraj Bo\v{z}i\'{c} who came up with it during his studies at Faculty of Architecture, University of Zagreb. His main goal was to design a model for settlement planning where the impact of the architect would be as small as possible and people would have a lot of freedom in the process of building the settlement. This minimal intervention from the side of the architect is given through the condition that houses are not allowed to be blocked from the sunlight and that the tracts of land on which the settlements are built are of rectangular shapes.

The rest of the paper is organized as follows. In Section \ref{sec:modelingproblem} we formalize our stochastic models. In Section \ref{sec:dist_X^s} we study the probability distribution of the building density of a random maximal configuration. Finally, in Section \ref{sec:packingDensity} we consider the asymptotics when the dimensions of the tract of land increase to infinity.

\section{Modeling of a problem}\label{sec:modelingproblem}
Let us denote with $\calM_{m, n}$ the set of all maximal configurations on an $m \times n$ tract of land. We sample at random one maximal configuration $C$ from $\calM_{m, n}$. The main questions that we are interested in are:
\begin{itemize}
    \item What is the distribution of the random variable that measures the building density of the sampled configuration $C$ (i.e.\ what is the distribution of $(mn)^{-1} \cdot \aps{C}$)?
    \item How does the expectation of this random variable behave when we let $m$ or $n$ or both ($m$ and $n$) to infinity?
\end{itemize}

\medskip

Clearly, to be able to answer these questions, we need to specify precisely how we are sampling at random from the set $\calM_{m, n}$. There are (at least) two natural ways of doing this:

\medskip

\textit{Model (i): Sampling uniformly at random}

\noindent In this model, we assign equal probabilities to all the maximal configurations in $\calM_{m, n}$ and we sample one of them uniformly at random (see e.g.\ Figures \ref{fig:M_33}, \ref{fig:M_34} and \ref{fig:M_43} for the illustration how the set $\calM_{m, n}$ can look like). Notice that we need to know all the maximal configurations to be able to analyze this model and this turns out to be quite involved. As already mentioned, we are interested in the random variable measuring the building density of such a uniformly sampled maximal configuration. Denoting this random variable with $X^u_{m, n}$, we have 
    \begin{equation*}
        \bbP\OBL{X_{m, n}^u = \frac{k}{mn}} = \frac{\aps{\{C \in \calM_{m, n} : \aps{C} = k\}}}{\aps{\calM_{m, n}}}, \qquad k \in \bbN.
    \end{equation*}
    
\medskip
    
\textit{Model (ii): Sequential building of houses}

\noindent
In this model, we build houses one-by-one. Let $C_r$ denote the set of occupied lots after the $r$-th step of the algorithm. Let $B_r$ denote the set of lots where we tried to build a house (in the first $r$ steps), but were not able to, because this house, or some other, already present house, would then become blocked from the sunlight.

We start with $C_0 = B_0 = \emptyset$. In the first step, we sample uniformly at random one element $(i, j)$ from the set $[m] \times [n]$, we build a house on the corresponding lot and we set $C_1 = \{(i, j)\}$ and still keep $B_1 = \emptyset$. In the $r$-th step, we again sample uniformly at random one element $(i, j)$, but this time from the set $([m] \times [n]) \setminus (C_{r - 1} \cup B_{r - 1})$. If, after building a house on the sampled lot, our configuration stays permissible, we set $C_r = C_{r - 1} \cup \{(i, j)\}$ and $B_r = B_{r - 1}$. Otherwise, we set $C_r = C_{r - 1}$ and $B_r = B_{r - 1} \cup \{(i, j)\}$. In the end (after $mn$ steps), we have $C_{mn} = C$, for some maximal configuration $C \in \calM_{m, n}$, and $B_{mn} = C^c$.

Notice that sampling random lots in this fashion results in one random permutation of the elements of the set $[m] \times [n]$. Therefore, we can assign a unique maximal configuration to any permutation of the set $[m] \times [n]$. Clearly, many different permutations will give us the same maximal configuration, see Figure \ref{fig:act_of_G}. Let $S_X$ be the set of all permutations of the elements of the set $X$ (recall that $\aps{S_X} = \aps{X}!$). Denote by $\FJADEF{G}{S_{[m] \times [n]}}{\calM_{m, n}}$ the function that maps a permutation in $S_{[m] \times [n]}$ to the corresponding maximal configuration in $\calM_{m, n}$, as explained above. In this model of sequential building of houses, we do not sample maximal configurations uniformly at random, instead we sample permutations from $S_{[m] \times [n]}$ uniformly at random and then consider the corresponding maximal configurations. Hence, denoting the random variable that measures the building density of a maximal configuration (that we got by random sampling a permutation $\sigma \in S_{[m] \times [n]}$) by $X_{m, n}^s$, we have
    \begin{equation*}
        \bbP\OBL{X_{m, n}^s = \frac{k}{mn}} = \frac{\aps{\{\sigma \in S_{[m] \times [n]} : \aps{G(\sigma)} = k\}}}{(mn)!}, \qquad k \in \bbN.
    \end{equation*}

\begin{figure}
    \begin{center}
        \begin{tabular}{ccc}
        $((1, 1), (1, 2), (1, 3), (2, 1), (2, 2), (2, 3), (3, 1), (3, 2), (3, 3))$ & \multirow{3}{*}{$\stackrel{G}{\longmapsto}$} & \multirow{3}{*}{
        \begin{tikzpicture}[scale = 0.5]
        \draw[step=1cm,black,very thin] (0, 0) grid (3,3);
        \fill[blue!40!white] (0,0) rectangle (1,1);
        \fill[blue!40!white] (0,1) rectangle (1,2);
        \fill[blue!40!white] (0,2) rectangle (1,3);
        \fill[blue!40!white] (1,0) rectangle (2,1);
        \fill[blue!40!white] (1,2) rectangle (2,3);
        \fill[blue!40!white] (2,0) rectangle (3,1);
        \fill[blue!40!white] (2,1) rectangle (3,2);
        \fill[blue!40!white] (2,2) rectangle (3,3);
        \draw[step=1cm,black,very thin] (0, 0) grid (3,3);
    \end{tikzpicture}} \\
        $((3, 2), (1, 3), (1, 2), (3, 3), (2, 3), (1, 1), (2, 2), (3, 1), (2, 1))$ & & \\
        $((1, 3), (3, 2), (2, 1), (1, 1), (3, 1), (1, 2), (2, 3), (2, 2), (3, 3))$ & & \\ [15pt]
        $((2, 1), (2, 2), (2, 3), (3, 1), (3, 2), (3, 3), (1, 1), (1, 3), (1, 2))$ & \multirow{3}{*}{$\stackrel{G}{\longmapsto}$} & \multirow{3}{*}{
        \begin{tikzpicture}[scale = 0.5]
        \draw[step=1cm,black,very thin] (0, 0) grid (3,3);
        \fill[blue!40!white] (0,0) rectangle (1,1);
        \fill[blue!40!white] (0,1) rectangle (1,2);
        \fill[blue!40!white] (0,2) rectangle (1,3);
        \fill[blue!40!white] (1,1) rectangle (2,2);
        \fill[blue!40!white] (2,0) rectangle (3,1);
        \fill[blue!40!white] (2,1) rectangle (3,2);
        \fill[blue!40!white] (2,2) rectangle (3,3);
        \draw[step=1cm,black,very thin] (0, 0) grid (3,3);
    \end{tikzpicture}} \\
        $((1, 1), (2, 3), (2, 1), (1, 3), (2, 2), (1, 2), (3, 2), (3, 3), (3, 1))$ & & \\
        $((1, 3), (2, 3), (3, 3), (2, 2), (2, 1), (1, 1), (1, 2), (3, 1), (3, 2))$ & & \\ [10pt]
    \end{tabular}
    \caption{Acting of the function $G$ on some permutations in $S_{[3]\times[3]}$.}
    \label{fig:act_of_G}
    \end{center}
\end{figure}
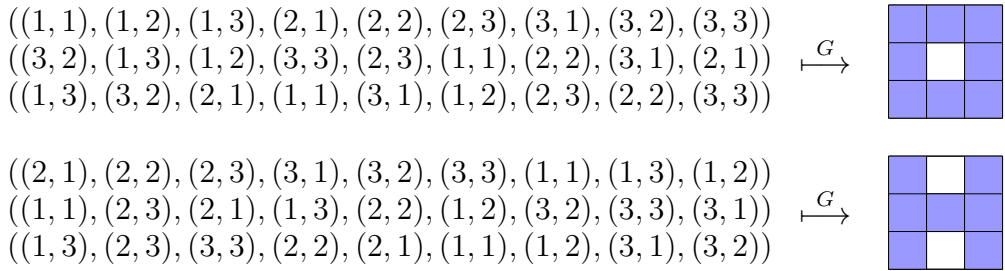

\medskip

A priori, it is not clear whether these two models are the same or not. It turns out that in general they are not the same. We illustrate this fact in the following claim.

\begin{claim}
    $X_{3, 3}^u \overset{(d)}{\neq} X_{3, 3}^s$.
\end{claim}
\begin{proof}
     Using the programming language R, we calculated the value of the function $G$ for all $(3\cdot 3)! = 362\,880$ permutations in $S_{[3]\times [3]}$. In this way we obtained all the maximal configurations in $\calM_{3, 3}$ (these are illustrated in Figure \ref{fig:M_33}) and, furthermore, we saw how many permutations are mapped to each of these maximal configurations. Even though there are $10$ maximal configurations in $\calM_{3, 3}$, it turns out that for each of the maximal configurations with $7$ occupied lots, there are only $25\,920$ permutations that the function $G$ maps to them, and there are $129\,600$ permutations that $G$ maps to the maximal configuration with $8$ occupied sites. Therefore
    \begin{equation*}
        X_{3, 3}^u \sim \begin{pmatrix}
            \frac{7}{9} & \frac{8}{9} \\[0.2cm]
            \frac{9}{10} & \frac{1}{10}
        \end{pmatrix}, \qquad X_{3, 3}^s \sim \begin{pmatrix}
            \frac{7}{9} & \frac{8}{9} \\[0.2cm]
            \frac{9}{14} & \frac{5}{14}
        \end{pmatrix}.
    \end{equation*}
\end{proof}


\begin{figure}
\begin{tabular}{ccccc}
    \begin{tikzpicture}[scale = 0.5]
        \draw[step=1cm,black,very thin] (0, 0) grid (3,3);
        \fill[blue!40!white] (0,0) rectangle (1,1);
        \fill[blue!40!white] (0,1) rectangle (1,2);
        \fill[blue!40!white] (0,2) rectangle (1,3);
        \fill[blue!40!white] (1,0) rectangle (2,1);
        \fill[blue!40!white] (1,2) rectangle (2,3);
        \fill[blue!40!white] (2,0) rectangle (3,1);
        \fill[blue!40!white] (2,1) rectangle (3,2);
        \fill[blue!40!white] (2,2) rectangle (3,3);
        \draw[step=1cm,black,very thin] (0, 0) grid (3,3);
    \end{tikzpicture}
    &
    \begin{tikzpicture}[scale = 0.5]
        \draw[step=1cm,black,very thin] (0, 0) grid (3,3);
        \fill[blue!40!white] (0,0) rectangle (1,1);
        \fill[blue!40!white] (1,0) rectangle (2,1);
        \fill[blue!40!white] (1,1) rectangle (2,2);
        \fill[blue!40!white] (1,2) rectangle (2,3);
        \fill[blue!40!white] (2,0) rectangle (3,1);
        \fill[blue!40!white] (2,1) rectangle (3,2);
        \fill[blue!40!white] (2,2) rectangle (3,3);
        \draw[step=1cm,black,very thin] (0, 0) grid (3,3);
    \end{tikzpicture}
    &
    \begin{tikzpicture}[scale = 0.5]
        \draw[step=1cm,black,very thin] (0, 0) grid (3,3);
        \fill[blue!40!white] (0,0) rectangle (1,1);
        \fill[blue!40!white] (0,1) rectangle (1,2);
        \fill[blue!40!white] (0,2) rectangle (1,3);
        \fill[blue!40!white] (1,0) rectangle (2,1);
        \fill[blue!40!white] (1,1) rectangle (2,2);
        \fill[blue!40!white] (1,2) rectangle (2,3);
        \fill[blue!40!white] (2,0) rectangle (3,1);
        \draw[step=1cm,black,very thin] (0, 0) grid (3,3);
    \end{tikzpicture}
    &
    \begin{tikzpicture}[scale = 0.5]
        \draw[step=1cm,black,very thin] (0, 0) grid (3,3);
        \fill[blue!40!white] (0,0) rectangle (1,1);
        \fill[blue!40!white] (0,2) rectangle (1,3);
        \fill[blue!40!white] (1,0) rectangle (2,1);
        \fill[blue!40!white] (1,1) rectangle (2,2);
        \fill[blue!40!white] (2,0) rectangle (3,1);
        \fill[blue!40!white] (2,1) rectangle (3,2);
        \fill[blue!40!white] (2,2) rectangle (3,3);
        \draw[step=1cm,black,very thin] (0, 0) grid (3,3);
    \end{tikzpicture}
    &
    \begin{tikzpicture}[scale = 0.5]
        \draw[step=1cm,black,very thin] (0, 0) grid (3,3);
        \fill[blue!40!white] (0,0) rectangle (1,1);
        \fill[blue!40!white] (0,1) rectangle (1,2);
        \fill[blue!40!white] (0,2) rectangle (1,3);
        \fill[blue!40!white] (1,0) rectangle (2,1);
        \fill[blue!40!white] (1,1) rectangle (2,2);
        \fill[blue!40!white] (2,0) rectangle (3,1);
        \fill[blue!40!white] (2,2) rectangle (3,3);
        \draw[step=1cm,black,very thin] (0, 0) grid (3,3);
    \end{tikzpicture} \\
(1) $129\,600$ & (2) $25\,920$ & (3) $25\,920$ & (4) $25\,920$ & (5) $25\,920$ \\[6pt]
    \begin{tikzpicture}[scale = 0.5]
        \draw[step=1cm,black,very thin] (0, 0) grid (3,3);
        \fill[blue!40!white] (0,0) rectangle (1,1);
        \fill[blue!40!white] (0,1) rectangle (1,2);
        \fill[blue!40!white] (0,2) rectangle (1,3);
        \fill[blue!40!white] (1,1) rectangle (2,2);
        \fill[blue!40!white] (2,0) rectangle (3,1);
        \fill[blue!40!white] (2,1) rectangle (3,2);
        \fill[blue!40!white] (2,2) rectangle (3,3);
        \draw[step=1cm,black,very thin] (0, 0) grid (3,3);
    \end{tikzpicture}
    &
    \begin{tikzpicture}[scale = 0.5]
        \draw[step=1cm,black,very thin] (0, 0) grid (3,3);
        \fill[blue!40!white] (0,0) rectangle (1,1);
        \fill[blue!40!white] (0,1) rectangle (1,2);
        \fill[blue!40!white] (1,1) rectangle (2,2);
        \fill[blue!40!white] (1,2) rectangle (2,3);
        \fill[blue!40!white] (2,0) rectangle (3,1);
        \fill[blue!40!white] (2,1) rectangle (3,2);
        \fill[blue!40!white] (2,2) rectangle (3,3);
        \draw[step=1cm,black,very thin] (0, 0) grid (3,3);
    \end{tikzpicture}
    &
    \begin{tikzpicture}[scale = 0.5]
        \draw[step=1cm,black,very thin] (0, 0) grid (3,3);
        \fill[blue!40!white] (0,0) rectangle (1,1);
        \fill[blue!40!white] (0,1) rectangle (1,2);
        \fill[blue!40!white] (0,2) rectangle (1,3);
        \fill[blue!40!white] (1,1) rectangle (2,2);
        \fill[blue!40!white] (1,2) rectangle (2,3);
        \fill[blue!40!white] (2,0) rectangle (3,1);
        \fill[blue!40!white] (2,1) rectangle (3,2);
        \draw[step=1cm,black,very thin] (0, 0) grid (3,3);
    \end{tikzpicture}
    &
    \begin{tikzpicture}[scale = 0.5]
        \draw[step=1cm,black,very thin] (0, 0) grid (3,3);
        \fill[blue!40!white] (0,0) rectangle (1,1);
        \fill[blue!40!white] (0,1) rectangle (1,2);
        \fill[blue!40!white] (1,0) rectangle (2,1);
        \fill[blue!40!white] (1,1) rectangle (2,2);
        \fill[blue!40!white] (1,2) rectangle (2,3);
        \fill[blue!40!white] (2,0) rectangle (3,1);
        \fill[blue!40!white] (2,2) rectangle (3,3);
        \draw[step=1cm,black,very thin] (0, 0) grid (3,3);
    \end{tikzpicture}
    &
    \begin{tikzpicture}[scale = 0.5]
        \draw[step=1cm,black,very thin] (0, 0) grid (3,3);
        \fill[blue!40!white] (0,0) rectangle (1,1);
        \fill[blue!40!white] (0,2) rectangle (1,3);
        \fill[blue!40!white] (1,0) rectangle (2,1);
        \fill[blue!40!white] (1,1) rectangle (2,2);
        \fill[blue!40!white] (1,2) rectangle (2,3);
        \fill[blue!40!white] (2,0) rectangle (3,1);
        \fill[blue!40!white] (2,1) rectangle (3,2);
        \draw[step=1cm,black,very thin] (0, 0) grid (3,3);
    \end{tikzpicture} \\
(6) $25\,920$ & (7) $25\,920$ & (8) $25\,920$ & (9) $25\,920$ & (10) $25\,920$
\end{tabular}
\caption{All the elements of $\calM_{3, 3}$ together with the number of permutations in $S_{[3]\times[3]}$ that are mapped to each by the function $G$.}\label{fig:M_33}
\end{figure}
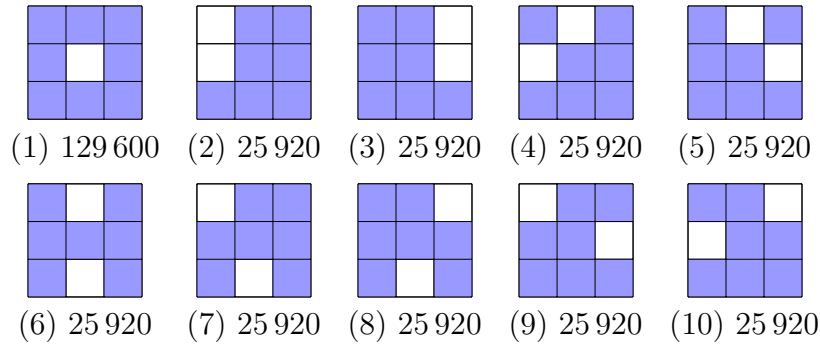


\begin{figure}
\begin{tabular}{ccccc}
    \begin{tikzpicture}[scale = 0.5]
        \draw[step=1cm,black,very thin] (0, 0) grid (4,3);
        \fill[blue!40!white] (0,0) rectangle (1,1);
        \fill[blue!40!white] (0,1) rectangle (1,2);
        \fill[blue!40!white] (1,0) rectangle (2,1);
        \fill[blue!40!white] (1,1) rectangle (2,2);
        \fill[blue!40!white] (1,2) rectangle (2,3);
        \fill[blue!40!white] (2,0) rectangle (3,1);
        \fill[blue!40!white] (2,2) rectangle (3,3);
        \fill[blue!40!white] (3,0) rectangle (4,1);
        \fill[blue!40!white] (3,1) rectangle (4,2);
        \fill[blue!40!white] (3,2) rectangle (4,3);
        \draw[step=1cm,black,very thin] (0, 0) grid (4,3);
    \end{tikzpicture}
    &
    \begin{tikzpicture}[scale = 0.5]
        \draw[step=1cm,black,very thin] (0, 0) grid (4,3);
        \fill[blue!40!white] (0,0) rectangle (1,1);
        \fill[blue!40!white] (0,1) rectangle (1,2);
        \fill[blue!40!white] (0,2) rectangle (1,3);
        \fill[blue!40!white] (1,0) rectangle (2,1);
        \fill[blue!40!white] (1,2) rectangle (2,3);
        \fill[blue!40!white] (2,0) rectangle (3,1);
        \fill[blue!40!white] (2,1) rectangle (3,2);
        \fill[blue!40!white] (2,2) rectangle (3,3);
        \fill[blue!40!white] (3,0) rectangle (4,1);
        \fill[blue!40!white] (3,1) rectangle (4,2);
        \draw[step=1cm,black,very thin] (0, 0) grid (4,3);
    \end{tikzpicture}
    &
    \begin{tikzpicture}[scale = 0.5]
        \draw[step=1cm,black,very thin] (0, 0) grid (4,3);
        \fill[blue!40!white] (0,0) rectangle (1,1);
        \fill[blue!40!white] (0,1) rectangle (1,2);
        \fill[blue!40!white] (0,2) rectangle (1,3);
        \fill[blue!40!white] (1,0) rectangle (2,1);
        \fill[blue!40!white] (2,0) rectangle (3,1);
        \fill[blue!40!white] (2,1) rectangle (3,2);
        \fill[blue!40!white] (2,2) rectangle (3,3);
        \fill[blue!40!white] (3,0) rectangle (4,1);
        \fill[blue!40!white] (3,1) rectangle (4,2);
        \fill[blue!40!white] (3,2) rectangle (4,3);
        \draw[step=1cm,black,very thin] (0, 0) grid (4,3);
    \end{tikzpicture}
    &
    \begin{tikzpicture}[scale = 0.5]
        \draw[step=1cm,black,very thin] (0, 0) grid (4,3);
        \fill[blue!40!white] (0,0) rectangle (1,1);
        \fill[blue!40!white] (0,1) rectangle (1,2);
        \fill[blue!40!white] (0,2) rectangle (1,3);
        \fill[blue!40!white] (1,0) rectangle (2,1);
        \fill[blue!40!white] (1,1) rectangle (2,2);
        \fill[blue!40!white] (1,2) rectangle (2,3);
        \fill[blue!40!white] (2,0) rectangle (3,1);
        \fill[blue!40!white] (3,0) rectangle (4,1);
        \fill[blue!40!white] (3,1) rectangle (4,2);
        \fill[blue!40!white] (3,2) rectangle (4,3);
        \draw[step=1cm,black,very thin] (0, 0) grid (4,3);
    \end{tikzpicture}
    &
    \begin{tikzpicture}[scale = 0.5]
        \draw[step=1cm,black,very thin] (0, 0) grid (4,3);
        \fill[blue!40!white] (0,0) rectangle (1,1);
        \fill[blue!40!white] (0,1) rectangle (1,2);
        \fill[blue!40!white] (0,2) rectangle (1,3);
        \fill[blue!40!white] (1,0) rectangle (2,1);
        \fill[blue!40!white] (1,1) rectangle (2,2);
        \fill[blue!40!white] (2,0) rectangle (3,1);
        \fill[blue!40!white] (2,2) rectangle (3,3);
        \fill[blue!40!white] (3,0) rectangle (4,1);
        \fill[blue!40!white] (3,1) rectangle (4,2);
        \fill[blue!40!white] (3,2) rectangle (4,3);
        \draw[step=1cm,black,very thin] (0, 0) grid (4,3);
    \end{tikzpicture}\\
(1) $56\,422\,080$ & (2) $56\,422\,080$ & (3) $46\,252\,800$ & (4) $46\,252\,800$ & (5) $46\,252\,800$ \\[6pt]
    \begin{tikzpicture}[scale = 0.5]
        \draw[step=1cm,black,very thin] (0, 0) grid (4,3);
        \fill[blue!40!white] (0,0) rectangle (1,1);
        \fill[blue!40!white] (0,1) rectangle (1,2);
        \fill[blue!40!white] (0,2) rectangle (1,3);
        \fill[blue!40!white] (1,0) rectangle (2,1);
        \fill[blue!40!white] (1,2) rectangle (2,3);
        \fill[blue!40!white] (2,0) rectangle (3,1);
        \fill[blue!40!white] (2,1) rectangle (3,2);
        \fill[blue!40!white] (3,0) rectangle (4,1);
        \fill[blue!40!white] (3,1) rectangle (4,2);
        \fill[blue!40!white] (3,2) rectangle (4,3);
        \draw[step=1cm,black,very thin] (0, 0) grid (4,3);
    \end{tikzpicture}
    &
    \begin{tikzpicture}[scale = 0.5]
        \draw[step=1cm,black,very thin] (0, 0) grid (4,3);
        \fill[blue!40!white] (0,0) rectangle (1,1);
        \fill[blue!40!white] (0,1) rectangle (1,2);
        \fill[blue!40!white] (0,2) rectangle (1,3);
        \fill[blue!40!white] (1,0) rectangle (2,1);
        \fill[blue!40!white] (1,2) rectangle (2,3);
        \fill[blue!40!white] (2,0) rectangle (3,1);
        \fill[blue!40!white] (2,2) rectangle (3,3);
        \fill[blue!40!white] (3,0) rectangle (4,1);
        \fill[blue!40!white] (3,1) rectangle (4,2);
        \fill[blue!40!white] (3,2) rectangle (4,3);
        \draw[step=1cm,black,very thin] (0, 0) grid (4,3);
    \end{tikzpicture}
    &
    \begin{tikzpicture}[scale = 0.5]
        \draw[step=1cm,black,very thin] (0, 0) grid (4,3);
        \fill[blue!40!white] (0,0) rectangle (1,1);
        \fill[blue!40!white] (0,1) rectangle (1,2);
        \fill[blue!40!white] (0,2) rectangle (1,3);
        \fill[blue!40!white] (1,1) rectangle (2,2);
        \fill[blue!40!white] (2,1) rectangle (3,2);
        \fill[blue!40!white] (2,2) rectangle (3,3);
        \fill[blue!40!white] (3,0) rectangle (4,1);
        \fill[blue!40!white] (3,1) rectangle (4,2);
        \fill[blue!40!white] (3,2) rectangle (4,3);
        \draw[step=1cm,black,very thin] (0, 0) grid (4,3);
    \end{tikzpicture}
    &
    \begin{tikzpicture}[scale = 0.5]
        \draw[step=1cm,black,very thin] (0, 0) grid (4,3);
        \fill[blue!40!white] (0,0) rectangle (1,1);
        \fill[blue!40!white] (0,1) rectangle (1,2);
        \fill[blue!40!white] (0,2) rectangle (1,3);
        \fill[blue!40!white] (1,1) rectangle (2,2);
        \fill[blue!40!white] (1,2) rectangle (2,3);
        \fill[blue!40!white] (2,1) rectangle (3,2);
        \fill[blue!40!white] (3,0) rectangle (4,1);
        \fill[blue!40!white] (3,1) rectangle (4,2);
        \fill[blue!40!white] (3,2) rectangle (4,3);
        \draw[step=1cm,black,very thin] (0, 0) grid (4,3);
    \end{tikzpicture}
    &
    \begin{tikzpicture}[scale = 0.5]
        \draw[step=1cm,black,very thin] (0, 0) grid (4,3);
        \fill[blue!40!white] (0,0) rectangle (1,1);
        \fill[blue!40!white] (0,2) rectangle (1,3);
        \fill[blue!40!white] (1,0) rectangle (2,1);
        \fill[blue!40!white] (1,1) rectangle (2,2);
        \fill[blue!40!white] (2,1) rectangle (3,2);
        \fill[blue!40!white] (2,2) rectangle (3,3);
        \fill[blue!40!white] (3,0) rectangle (4,1);
        \fill[blue!40!white] (3,1) rectangle (4,2);
        \fill[blue!40!white] (3,2) rectangle (4,3);
        \draw[step=1cm,black,very thin] (0, 0) grid (4,3);
    \end{tikzpicture} \\
(6) $46\,252\,800$ & (7) $39\,916\,800$ & (8) $15\,095\,520$ & (9) $15\,095\,520$ & (10) $15\,095\,520$ \\[6pt]
    \begin{tikzpicture}[scale = 0.5]
        \draw[step=1cm,black,very thin] (0, 0) grid (4,3);
        \fill[blue!40!white] (0,0) rectangle (1,1);
        \fill[blue!40!white] (0,1) rectangle (1,2);
        \fill[blue!40!white] (0,2) rectangle (1,3);
        \fill[blue!40!white] (1,1) rectangle (2,2);
        \fill[blue!40!white] (1,2) rectangle (2,3);
        \fill[blue!40!white] (2,0) rectangle (3,1);
        \fill[blue!40!white] (2,1) rectangle (3,2);
        \fill[blue!40!white] (3,0) rectangle (4,1);
        \fill[blue!40!white] (3,2) rectangle (4,3);
        \draw[step=1cm,black,very thin] (0, 0) grid (4,3);
    \end{tikzpicture}
    &
    \begin{tikzpicture}[scale = 0.5]
        \draw[step=1cm,black,very thin] (0, 0) grid (4,3);
        \fill[blue!40!white] (0,0) rectangle (1,1);
        \fill[blue!40!white] (0,1) rectangle (1,2);
        \fill[blue!40!white] (0,2) rectangle (1,3);
        \fill[blue!40!white] (1,1) rectangle (2,2);
        \fill[blue!40!white] (2,0) rectangle (3,1);
        \fill[blue!40!white] (2,1) rectangle (3,2);
        \fill[blue!40!white] (2,2) rectangle (3,3);
        \fill[blue!40!white] (3,0) rectangle (4,1);
        \fill[blue!40!white] (3,2) rectangle (4,3);
        \draw[step=1cm,black,very thin] (0, 0) grid (4,3);
    \end{tikzpicture}
    &
    \begin{tikzpicture}[scale = 0.5]
        \draw[step=1cm,black,very thin] (0, 0) grid (4,3);
        \fill[blue!40!white] (0,0) rectangle (1,1);
        \fill[blue!40!white] (0,2) rectangle (1,3);
        \fill[blue!40!white] (1,0) rectangle (2,1);
        \fill[blue!40!white] (1,1) rectangle (2,2);
        \fill[blue!40!white] (1,2) rectangle (2,3);
        \fill[blue!40!white] (2,1) rectangle (3,2);
        \fill[blue!40!white] (3,0) rectangle (4,1);
        \fill[blue!40!white] (3,1) rectangle (4,2);
        \fill[blue!40!white] (3,2) rectangle (4,3);
        \draw[step=1cm,black,very thin] (0, 0) grid (4,3);
    \end{tikzpicture}
    &
    \begin{tikzpicture}[scale = 0.5]
        \draw[step=1cm,black,very thin] (0, 0) grid (4,3);
        \fill[blue!40!white] (0,0) rectangle (1,1);
        \fill[blue!40!white] (0,2) rectangle (1,3);
        \fill[blue!40!white] (1,0) rectangle (2,1);
        \fill[blue!40!white] (1,1) rectangle (2,2);
        \fill[blue!40!white] (2,0) rectangle (3,1);
        \fill[blue!40!white] (2,1) rectangle (3,2);
        \fill[blue!40!white] (2,2) rectangle (3,3);
        \fill[blue!40!white] (3,0) rectangle (4,1);
        \fill[blue!40!white] (3,2) rectangle (4,3);
        \draw[step=1cm,black,very thin] (0, 0) grid (4,3);
    \end{tikzpicture}
    &
    \begin{tikzpicture}[scale = 0.5]
        \draw[step=1cm,black,very thin] (0, 0) grid (4,3);
        \fill[blue!40!white] (0,0) rectangle (1,1);
        \fill[blue!40!white] (0,2) rectangle (1,3);
        \fill[blue!40!white] (1,0) rectangle (2,1);
        \fill[blue!40!white] (1,1) rectangle (2,2);
        \fill[blue!40!white] (1,2) rectangle (2,3);
        \fill[blue!40!white] (2,0) rectangle (3,1);
        \fill[blue!40!white] (2,1) rectangle (3,2);
        \fill[blue!40!white] (3,0) rectangle (4,1);
        \fill[blue!40!white] (3,2) rectangle (4,3);
        \draw[step=1cm,black,very thin] (0, 0) grid (4,3);
    \end{tikzpicture} \\
(11) $15\,095\,520$ & (12) $15\,095\,520$ & (13) $15\,095\,520$ & (14) $15\,095\,520$ & (15) $15\,095\,520$ \\[6pt]
    \begin{tikzpicture}[scale = 0.5]
        \draw[step=1cm,black,very thin] (0, 0) grid (4,3);
        \fill[blue!40!white] (0,0) rectangle (1,1);
        \fill[blue!40!white] (0,1) rectangle (1,2);
        \fill[blue!40!white] (1,1) rectangle (2,2);
        \fill[blue!40!white] (1,2) rectangle (2,3);
        \fill[blue!40!white] (2,1) rectangle (3,2);
        \fill[blue!40!white] (2,2) rectangle (3,3);
        \fill[blue!40!white] (3,0) rectangle (4,1);
        \fill[blue!40!white] (3,1) rectangle (4,2);
        \draw[step=1cm,black,very thin] (0, 0) grid (4,3);
    \end{tikzpicture}
    &
    \begin{tikzpicture}[scale = 0.5]
        \draw[step=1cm,black,very thin] (0, 0) grid (4,3);
        \fill[blue!40!white] (0,0) rectangle (1,1);
        \fill[blue!40!white] (1,0) rectangle (2,1);
        \fill[blue!40!white] (1,1) rectangle (2,2);
        \fill[blue!40!white] (1,2) rectangle (2,3);
        \fill[blue!40!white] (2,0) rectangle (3,1);
        \fill[blue!40!white] (2,1) rectangle (3,2);
        \fill[blue!40!white] (2,2) rectangle (3,3);
        \fill[blue!40!white] (3,0) rectangle (4,1);
        \draw[step=1cm,black,very thin] (0, 0) grid (4,3);
    \end{tikzpicture}
    &
    \begin{tikzpicture}[scale = 0.5]
        \draw[step=1cm,black,very thin] (0, 0) grid (4,3);
        \fill[blue!40!white] (0,0) rectangle (1,1);
        \fill[blue!40!white] (1,0) rectangle (2,1);
        \fill[blue!40!white] (1,1) rectangle (2,2);
        \fill[blue!40!white] (1,2) rectangle (2,3);
        \fill[blue!40!white] (2,1) rectangle (3,2);
        \fill[blue!40!white] (2,2) rectangle (3,3);
        \fill[blue!40!white] (3,0) rectangle (4,1);
        \fill[blue!40!white] (3,1) rectangle (4,2);
        \draw[step=1cm,black,very thin] (0, 0) grid (4,3);
    \end{tikzpicture}
    &
    \begin{tikzpicture}[scale = 0.5]
        \draw[step=1cm,black,very thin] (0, 0) grid (4,3);
        \fill[blue!40!white] (0,0) rectangle (1,1);
        \fill[blue!40!white] (0,1) rectangle (1,2);
        \fill[blue!40!white] (1,1) rectangle (2,2);
        \fill[blue!40!white] (1,2) rectangle (2,3);
        \fill[blue!40!white] (2,0) rectangle (3,1);
        \fill[blue!40!white] (2,1) rectangle (3,2);
        \fill[blue!40!white] (2,2) rectangle (3,3);
        \fill[blue!40!white] (3,0) rectangle (4,1);
        \draw[step=1cm,black,very thin] (0, 0) grid (4,3);
    \end{tikzpicture}
    &
     \\
(16) $5\,116\,320$ & (17) $5\,116\,320$ & (18) $5\,116\,320$ & (19) $5\,116\,320$ &  
\end{tabular}
\caption{All the elements of $\calM_{3, 4}$ together with the number of permutations in $S_{[3]\times [4]}$ that are mapped to each by the function $G$.}\label{fig:M_34}
\end{figure}
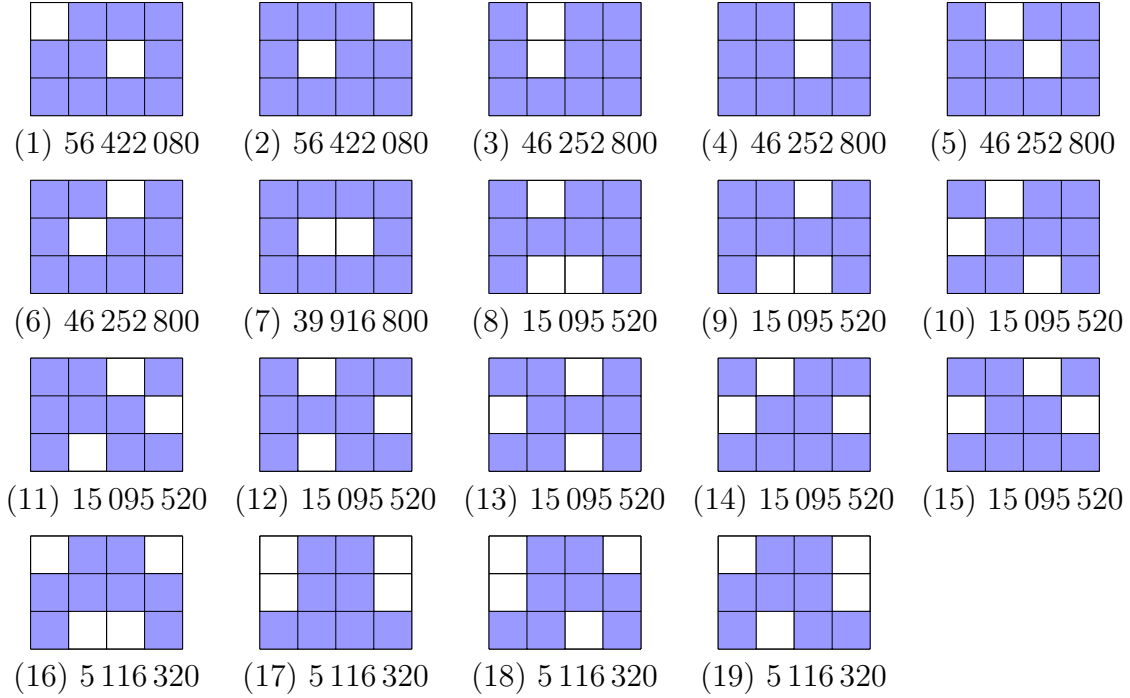


\begin{figure}
\begin{tabular}{ccccc}
    \begin{tikzpicture}[scale = 0.5]
        \draw[step=1cm,black,very thin] (0, 0) grid (3,4);
        \fill[blue!40!white] (0,0) rectangle (1,1);
        \fill[blue!40!white] (0,1) rectangle (1,2);
        \fill[blue!40!white] (0,2) rectangle (1,3);
        \fill[blue!40!white] (0,3) rectangle (1,4);
        \fill[blue!40!white] (1,0) rectangle (2,1);
        \fill[blue!40!white] (1,2) rectangle (2,3);
        \fill[blue!40!white] (2,0) rectangle (3,1);
        \fill[blue!40!white] (2,1) rectangle (3,2);
        \fill[blue!40!white] (2,2) rectangle (3,3);
        \fill[blue!40!white] (2,3) rectangle (3,4);
        \draw[step=1cm,black,very thin] (0, 0) grid (3,4);
    \end{tikzpicture}
    &
    \begin{tikzpicture}[scale = 0.5]
        \draw[step=1cm,black,very thin] (0, 0) grid (3,4);
        \fill[blue!40!white] (0,0) rectangle (1,1);
        \fill[blue!40!white] (0,1) rectangle (1,2);
        \fill[blue!40!white] (0,2) rectangle (1,3);
        \fill[blue!40!white] (0,3) rectangle (1,4);
        \fill[blue!40!white] (1,1) rectangle (2,2);
        \fill[blue!40!white] (1,3) rectangle (2,4);
        \fill[blue!40!white] (2,0) rectangle (3,1);
        \fill[blue!40!white] (2,1) rectangle (3,2);
        \fill[blue!40!white] (2,2) rectangle (3,3);
        \fill[blue!40!white] (2,3) rectangle (3,4);
        \draw[step=1cm,black,very thin] (0, 0) grid (3,4);
    \end{tikzpicture}
    &
    \begin{tikzpicture}[scale = 0.5]
        \draw[step=1cm,black,very thin] (0, 0) grid (3,4);
        \fill[blue!40!white] (0,0) rectangle (1,1);
        \fill[blue!40!white] (0,2) rectangle (1,3);
        \fill[blue!40!white] (0,3) rectangle (1,4);
        \fill[blue!40!white] (1,0) rectangle (2,1);
        \fill[blue!40!white] (1,1) rectangle (2,2);
        \fill[blue!40!white] (1,3) rectangle (2,4);
        \fill[blue!40!white] (2,0) rectangle (3,1);
        \fill[blue!40!white] (2,1) rectangle (3,2);
        \fill[blue!40!white] (2,2) rectangle (3,3);
        \fill[blue!40!white] (2,3) rectangle (3,4);
        \draw[step=1cm,black,very thin] (0, 0) grid (3,4);
    \end{tikzpicture}
    &
    \begin{tikzpicture}[scale = 0.5]
        \draw[step=1cm,black,very thin] (0, 0) grid (3,4);
        \fill[blue!40!white] (0,0) rectangle (1,1);
        \fill[blue!40!white] (0,1) rectangle (1,2);
        \fill[blue!40!white] (0,2) rectangle (1,3);
        \fill[blue!40!white] (0,3) rectangle (1,4);
        \fill[blue!40!white] (1,0) rectangle (2,1);
        \fill[blue!40!white] (1,1) rectangle (2,2);
        \fill[blue!40!white] (1,3) rectangle (2,4);
        \fill[blue!40!white] (2,0) rectangle (3,1);
        \fill[blue!40!white] (2,2) rectangle (3,3);
        \fill[blue!40!white] (2,3) rectangle (3,4);
        \draw[step=1cm,black,very thin] (0, 0) grid (3,4);
    \end{tikzpicture}
    &
    \begin{tikzpicture}[scale = 0.5]
        \draw[step=1cm,black,very thin] (0, 0) grid (3,4);
        \fill[blue!40!white] (0,0) rectangle (1,1);
        \fill[blue!40!white] (0,1) rectangle (1,2);
        \fill[blue!40!white] (0,2) rectangle (1,3);
        \fill[blue!40!white] (1,0) rectangle (2,1);
        \fill[blue!40!white] (1,2) rectangle (2,3);
        \fill[blue!40!white] (1,3) rectangle (2,4);
        \fill[blue!40!white] (2,0) rectangle (3,1);
        \fill[blue!40!white] (2,1) rectangle (3,2);
        \fill[blue!40!white] (2,2) rectangle (3,3);
        \fill[blue!40!white] (2,3) rectangle (3,4);
        \draw[step=1cm,black,very thin] (0, 0) grid (3,4);
    \end{tikzpicture}\\
(1) $45\,334\,080$ & (2) $45\,334\,080$ & (3) $45\,334\,080$ & (4) $45\,334\,080$ & (5) $45\,334\,080$ \\[6pt]
    \begin{tikzpicture}[scale = 0.5]
        \draw[step=1cm,black,very thin] (0, 0) grid (3,4);
        \fill[blue!40!white] (0,0) rectangle (1,1);
        \fill[blue!40!white] (0,1) rectangle (1,2);
        \fill[blue!40!white] (0,2) rectangle (1,3);
        \fill[blue!40!white] (0,3) rectangle (1,4);
        \fill[blue!40!white] (1,0) rectangle (2,1);
        \fill[blue!40!white] (1,2) rectangle (2,3);
        \fill[blue!40!white] (1,3) rectangle (2,4);
        \fill[blue!40!white] (2,0) rectangle (3,1);
        \fill[blue!40!white] (2,1) rectangle (3,2);
        \fill[blue!40!white] (2,2) rectangle (3,3);
        \draw[step=1cm,black,very thin] (0, 0) grid (3,4);
    \end{tikzpicture}
    &
    \begin{tikzpicture}[scale = 0.5]
        \draw[step=1cm,black,very thin] (0, 0) grid (3,4);
        \fill[blue!40!white] (0,0) rectangle (1,1);
        \fill[blue!40!white] (0,1) rectangle (1,2);
        \fill[blue!40!white] (0,2) rectangle (1,3);
        \fill[blue!40!white] (0,3) rectangle (1,4);
        \fill[blue!40!white] (1,0) rectangle (2,1);
        \fill[blue!40!white] (1,3) rectangle (2,4);
        \fill[blue!40!white] (2,0) rectangle (3,1);
        \fill[blue!40!white] (2,1) rectangle (3,2);
        \fill[blue!40!white] (2,2) rectangle (3,3);
        \fill[blue!40!white] (2,3) rectangle (3,4);
        \draw[step=1cm,black,very thin] (0, 0) grid (3,4);
    \end{tikzpicture}
    &
    \begin{tikzpicture}[scale = 0.5]
        \draw[step=1cm,black,very thin] (0, 0) grid (3,4);
        \fill[blue!40!white] (0,0) rectangle (1,1);
        \fill[blue!40!white] (0,1) rectangle (1,2);
        \fill[blue!40!white] (1,1) rectangle (2,2);
        \fill[blue!40!white] (1,2) rectangle (2,3);
        \fill[blue!40!white] (1,3) rectangle (2,4);
        \fill[blue!40!white] (2,0) rectangle (3,1);
        \fill[blue!40!white] (2,1) rectangle (3,2);
        \fill[blue!40!white] (2,2) rectangle (3,3);
        \fill[blue!40!white] (2,3) rectangle (3,4);
        \draw[step=1cm,black,very thin] (0, 0) grid (3,4);
    \end{tikzpicture}
    &
    \begin{tikzpicture}[scale = 0.5]
        \draw[step=1cm,black,very thin] (0, 0) grid (3,4);
        \fill[blue!40!white] (0,0) rectangle (1,1);
        \fill[blue!40!white] (0,1) rectangle (1,2);
        \fill[blue!40!white] (0,2) rectangle (1,3);
        \fill[blue!40!white] (0,3) rectangle (1,4);
        \fill[blue!40!white] (1,1) rectangle (2,2);
        \fill[blue!40!white] (1,2) rectangle (2,3);
        \fill[blue!40!white] (1,3) rectangle (2,4);
        \fill[blue!40!white] (2,0) rectangle (3,1);
        \fill[blue!40!white] (2,1) rectangle (3,2);
        \draw[step=1cm,black,very thin] (0, 0) grid (3,4);
    \end{tikzpicture}
    &
    \begin{tikzpicture}[scale = 0.5]
        \draw[step=1cm,black,very thin] (0, 0) grid (3,4);
        \fill[blue!40!white] (0,0) rectangle (1,1);
        \fill[blue!40!white] (1,0) rectangle (2,1);
        \fill[blue!40!white] (1,1) rectangle (2,2);
        \fill[blue!40!white] (1,2) rectangle (2,3);
        \fill[blue!40!white] (1,3) rectangle (2,4);
        \fill[blue!40!white] (2,0) rectangle (3,1);
        \fill[blue!40!white] (2,1) rectangle (3,2);
        \fill[blue!40!white] (2,2) rectangle (3,3);
        \fill[blue!40!white] (2,3) rectangle (3,4);
        \draw[step=1cm,black,very thin] (0, 0) grid (3,4);
    \end{tikzpicture} \\
(6) $45\,334\,080$ & (7) $29\,937\,600$ & (8) $9\,836\,640$ & (9) $9\,836\,640$ & (10) $9\,836\,640$ \\[6pt]
    \begin{tikzpicture}[scale = 0.5]
        \draw[step=1cm,black,very thin] (0, 0) grid (3,4);
        \fill[blue!40!white] (0,0) rectangle (1,1);
        \fill[blue!40!white] (0,1) rectangle (1,2);
        \fill[blue!40!white] (0,2) rectangle (1,3);
        \fill[blue!40!white] (0,3) rectangle (1,4);
        \fill[blue!40!white] (1,0) rectangle (2,1);
        \fill[blue!40!white] (1,1) rectangle (2,2);
        \fill[blue!40!white] (1,2) rectangle (2,3);
        \fill[blue!40!white] (1,3) rectangle (2,4);
        \fill[blue!40!white] (2,0) rectangle (3,1);
        \draw[step=1cm,black,very thin] (0, 0) grid (3,4);
    \end{tikzpicture}
    &
    \begin{tikzpicture}[scale = 0.5]
        \draw[step=1cm,black,very thin] (0, 0) grid (3,4);
        \fill[blue!40!white] (0,0) rectangle (1,1);
        \fill[blue!40!white] (0,1) rectangle (1,2);
        \fill[blue!40!white] (1,0) rectangle (2,1);
        \fill[blue!40!white] (1,1) rectangle (2,2);
        \fill[blue!40!white] (1,2) rectangle (2,3);
        \fill[blue!40!white] (1,3) rectangle (2,4);
        \fill[blue!40!white] (2,0) rectangle (3,1);
        \fill[blue!40!white] (2,2) rectangle (3,3);
        \fill[blue!40!white] (2,3) rectangle (3,4);
        \draw[step=1cm,black,very thin] (0, 0) grid (3,4);
    \end{tikzpicture}
    &
    \begin{tikzpicture}[scale = 0.5]
        \draw[step=1cm,black,very thin] (0, 0) grid (3,4);
        \fill[blue!40!white] (0,0) rectangle (1,1);
        \fill[blue!40!white] (0,2) rectangle (1,3);
        \fill[blue!40!white] (0,3) rectangle (1,4);
        \fill[blue!40!white] (1,0) rectangle (2,1);
        \fill[blue!40!white] (1,1) rectangle (2,2);
        \fill[blue!40!white] (1,2) rectangle (2,3);
        \fill[blue!40!white] (1,3) rectangle (2,4);
        \fill[blue!40!white] (2,0) rectangle (3,1);
        \fill[blue!40!white] (2,1) rectangle (3,2);
        \draw[step=1cm,black,very thin] (0, 0) grid (3,4);
    \end{tikzpicture}
    &
    \begin{tikzpicture}[scale = 0.5]
        \draw[step=1cm,black,very thin] (0, 0) grid (3,4);
        \fill[blue!40!white] (0,0) rectangle (1,1);
        \fill[blue!40!white] (0,1) rectangle (1,2);
        \fill[blue!40!white] (0,2) rectangle (1,3);
        \fill[blue!40!white] (1,1) rectangle (2,2);
        \fill[blue!40!white] (1,2) rectangle (2,3);
        \fill[blue!40!white] (1,3) rectangle (2,4);
        \fill[blue!40!white] (2,0) rectangle (3,1);
        \fill[blue!40!white] (2,1) rectangle (3,2);
        \fill[blue!40!white] (2,3) rectangle (3,4);
        \draw[step=1cm,black,very thin] (0, 0) grid (3,4);
    \end{tikzpicture}
    &
    \begin{tikzpicture}[scale = 0.5]
        \draw[step=1cm,black,very thin] (0, 0) grid (3,4);
        \fill[blue!40!white] (0,0) rectangle (1,1);
        \fill[blue!40!white] (0,1) rectangle (1,2);
        \fill[blue!40!white] (0,3) rectangle (1,4);
        \fill[blue!40!white] (1,1) rectangle (2,2);
        \fill[blue!40!white] (1,2) rectangle (2,3);
        \fill[blue!40!white] (1,3) rectangle (2,4);
        \fill[blue!40!white] (2,0) rectangle (3,1);
        \fill[blue!40!white] (2,1) rectangle (3,2);
        \fill[blue!40!white] (2,2) rectangle (3,3);
        \draw[step=1cm,black,very thin] (0, 0) grid (3,4);
    \end{tikzpicture} \\
(11) $9\,836\,640$ & (12) $9\,836\,640$ & (13) $9\,836\,640$ & (14) $9\,836\,640$ & (15) $9\,836\,640$ \\[6pt]
    \begin{tikzpicture}[scale = 0.5]
        \draw[step=1cm,black,very thin] (0, 0) grid (3,4);
        \fill[blue!40!white] (0,0) rectangle (1,1);
        \fill[blue!40!white] (0,2) rectangle (1,3);
        \fill[blue!40!white] (1,0) rectangle (2,1);
        \fill[blue!40!white] (1,1) rectangle (2,2);
        \fill[blue!40!white] (1,2) rectangle (2,3);
        \fill[blue!40!white] (1,3) rectangle (2,4);
        \fill[blue!40!white] (2,0) rectangle (3,1);
        \fill[blue!40!white] (2,1) rectangle (3,2);
        \fill[blue!40!white] (2,3) rectangle (3,4);
        \draw[step=1cm,black,very thin] (0, 0) grid (3,4);
    \end{tikzpicture}
    &
    \begin{tikzpicture}[scale = 0.5]
        \draw[step=1cm,black,very thin] (0, 0) grid (3,4);
        \fill[blue!40!white] (0,0) rectangle (1,1);
        \fill[blue!40!white] (0,1) rectangle (1,2);
        \fill[blue!40!white] (0,3) rectangle (1,4);
        \fill[blue!40!white] (1,0) rectangle (2,1);
        \fill[blue!40!white] (1,1) rectangle (2,2);
        \fill[blue!40!white] (1,2) rectangle (2,3);
        \fill[blue!40!white] (1,3) rectangle (2,4);
        \fill[blue!40!white] (2,0) rectangle (3,1);
        \fill[blue!40!white] (2,2) rectangle (3,3);
        \draw[step=1cm,black,very thin] (0, 0) grid (3,4);
    \end{tikzpicture}
    &
    \begin{tikzpicture}[scale = 0.5]
        \draw[step=1cm,black,very thin] (0, 0) grid (3,4);
        \fill[blue!40!white] (0,0) rectangle (1,1);
        \fill[blue!40!white] (0,1) rectangle (1,2);
        \fill[blue!40!white] (0,2) rectangle (1,3);
        \fill[blue!40!white] (1,0) rectangle (2,1);
        \fill[blue!40!white] (1,1) rectangle (2,2);
        \fill[blue!40!white] (1,2) rectangle (2,3);
        \fill[blue!40!white] (1,3) rectangle (2,4);
        \fill[blue!40!white] (2,0) rectangle (3,1);
        \fill[blue!40!white] (2,3) rectangle (3,4);
        \draw[step=1cm,black,very thin] (0, 0) grid (3,4);
    \end{tikzpicture}
    &
    \begin{tikzpicture}[scale = 0.5]
        \draw[step=1cm,black,very thin] (0, 0) grid (3,4);
        \fill[blue!40!white] (0,0) rectangle (1,1);
        \fill[blue!40!white] (0,3) rectangle (1,4);
        \fill[blue!40!white] (1,0) rectangle (2,1);
        \fill[blue!40!white] (1,1) rectangle (2,2);
        \fill[blue!40!white] (1,2) rectangle (2,3);
        \fill[blue!40!white] (1,3) rectangle (2,4);
        \fill[blue!40!white] (2,0) rectangle (3,1);
        \fill[blue!40!white] (2,1) rectangle (3,2);
        \fill[blue!40!white] (2,2) rectangle (3,3);
        \draw[step=1cm,black,very thin] (0, 0) grid (3,4);
    \end{tikzpicture}
    &
    \begin{tikzpicture}[scale = 0.5]
        \draw[step=1cm,black,very thin] (0, 0) grid (3,4);
        \fill[blue!40!white] (0,0) rectangle (1,1);
        \fill[blue!40!white] (0,1) rectangle (1,2);
        \fill[blue!40!white] (0,3) rectangle (1,4);
        \fill[blue!40!white] (1,1) rectangle (2,2);
        \fill[blue!40!white] (1,2) rectangle (2,3);
        \fill[blue!40!white] (2,0) rectangle (3,1);
        \fill[blue!40!white] (2,1) rectangle (3,2);
        \fill[blue!40!white] (2,2) rectangle (3,3);
        \fill[blue!40!white] (2,3) rectangle (3,4);
        \draw[step=1cm,black,very thin] (0, 0) grid (3,4);
    \end{tikzpicture} \\
(16) $9\,836\,640$ & (17) $9\,836\,640$ & (18) $9\,836\,640$ & (19) $9\,836\,640$ & (20) $9\,836\,640$ \\[6pt]
    \begin{tikzpicture}[scale = 0.5]
        \draw[step=1cm,black,very thin] (0, 0) grid (3,4);
        \fill[blue!40!white] (0,0) rectangle (1,1);
        \fill[blue!40!white] (0,1) rectangle (1,2);
        \fill[blue!40!white] (0,2) rectangle (1,3);
        \fill[blue!40!white] (0,3) rectangle (1,4);
        \fill[blue!40!white] (1,1) rectangle (2,2);
        \fill[blue!40!white] (1,2) rectangle (2,3);
        \fill[blue!40!white] (2,0) rectangle (3,1);
        \fill[blue!40!white] (2,1) rectangle (3,2);
        \fill[blue!40!white] (2,3) rectangle (3,4);
        \draw[step=1cm,black,very thin] (0, 0) grid (3,4);
    \end{tikzpicture}
    &
    \begin{tikzpicture}[scale = 0.5]
        \draw[step=1cm,black,very thin] (0, 0) grid (3,4);
        \fill[blue!40!white] (0,0) rectangle (1,1);
        \fill[blue!40!white] (0,3) rectangle (1,4);
        \fill[blue!40!white] (1,0) rectangle (2,1);
        \fill[blue!40!white] (1,1) rectangle (2,2);
        \fill[blue!40!white] (1,2) rectangle (2,3);
        \fill[blue!40!white] (2,0) rectangle (3,1);
        \fill[blue!40!white] (2,1) rectangle (3,2);
        \fill[blue!40!white] (2,2) rectangle (3,3);
        \fill[blue!40!white] (2,3) rectangle (3,4);
        \draw[step=1cm,black,very thin] (0, 0) grid (3,4);
    \end{tikzpicture}
    &
    \begin{tikzpicture}[scale = 0.5]
        \draw[step=1cm,black,very thin] (0, 0) grid (3,4);
        \fill[blue!40!white] (0,0) rectangle (1,1);
        \fill[blue!40!white] (0,1) rectangle (1,2);
        \fill[blue!40!white] (0,2) rectangle (1,3);
        \fill[blue!40!white] (0,3) rectangle (1,4);
        \fill[blue!40!white] (1,0) rectangle (2,1);
        \fill[blue!40!white] (1,1) rectangle (2,2);
        \fill[blue!40!white] (1,2) rectangle (2,3);
        \fill[blue!40!white] (2,0) rectangle (3,1);
        \fill[blue!40!white] (2,3) rectangle (3,4);
        \draw[step=1cm,black,very thin] (0, 0) grid (3,4);
    \end{tikzpicture}
    &
    \begin{tikzpicture}[scale = 0.5]
        \draw[step=1cm,black,very thin] (0, 0) grid (3,4);
        \fill[blue!40!white] (0,0) rectangle (1,1);
        \fill[blue!40!white] (0,1) rectangle (1,2);
        \fill[blue!40!white] (0,3) rectangle (1,4);
        \fill[blue!40!white] (1,0) rectangle (2,1);
        \fill[blue!40!white] (1,1) rectangle (2,2);
        \fill[blue!40!white] (1,2) rectangle (2,3);
        \fill[blue!40!white] (2,0) rectangle (3,1);
        \fill[blue!40!white] (2,2) rectangle (3,3);
        \fill[blue!40!white] (2,3) rectangle (3,4);
        \draw[step=1cm,black,very thin] (0, 0) grid (3,4);
    \end{tikzpicture}
    &
    \begin{tikzpicture}[scale = 0.5]
        \draw[step=1cm,black,very thin] (0, 0) grid (3,4);
        \fill[blue!40!white] (0,0) rectangle (1,1);
        \fill[blue!40!white] (0,2) rectangle (1,3);
        \fill[blue!40!white] (0,3) rectangle (1,4);
        \fill[blue!40!white] (1,0) rectangle (2,1);
        \fill[blue!40!white] (1,1) rectangle (2,2);
        \fill[blue!40!white] (1,2) rectangle (2,3);
        \fill[blue!40!white] (2,0) rectangle (3,1);
        \fill[blue!40!white] (2,1) rectangle (3,2);
        \fill[blue!40!white] (2,3) rectangle (3,4);
        \draw[step=1cm,black,very thin] (0, 0) grid (3,4);
    \end{tikzpicture} \\
(21) $9\,836\,640$ & (22) $9\,836\,640$ & (23) $9\,836\,640$ & (24) $9\,836\,640$ & (25) $9\,836\,640$
\end{tabular}
\caption{All the elements of $\calM_{4, 3}$ together with the number of permutations in $S_{[4]\times [3]}$ that are mapped to each by the function $G$.}\label{fig:M_43}
\end{figure}
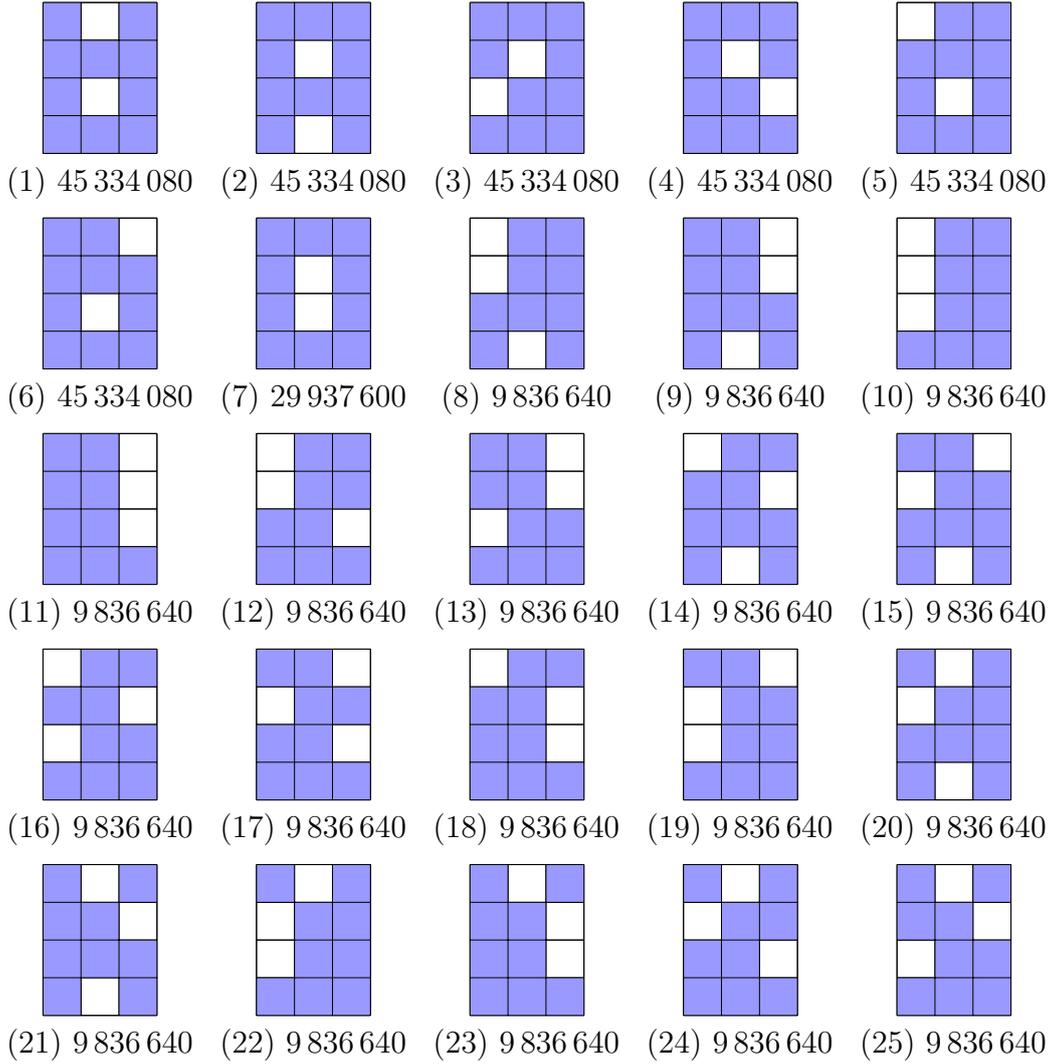

Using the same code in R, we found all the elements of sets $\calM_{3, 4}$ and $\calM_{4, 3}$, together with the number of permutations in $S_{[3]\times [4]}$ and $S_{[4]\times [3]}$, respectively, that are mapped to each of them (see Figures \ref{fig:M_34} and \ref{fig:M_43}). Inspecting the results, we see that for these dimensions of the tract of land, we also have
\begin{equation*}
    X_{3, 4}^u \overset{(d)}{\neq} X_{3, 4}^s \qquad \textnormal{and} \qquad X_{4, 3}^u \overset{(d)}{\neq} X_{4, 3}^s.
\end{equation*}
Furthermore, we can see that
\begin{equation*}
    X_{3, 4}^u \overset{(d)}{\neq} X_{4, 3}^u \qquad \textnormal{and} \qquad X_{3, 4}^s \overset{(d)}{\neq} X_{4, 3}^s.
\end{equation*}
Not only these variables do not have the same distribution, they do not even have the same support. Therefore, orientation of the tract of land is important. The reason for that is the difference in the role of north and south. Notice that east and west are symmetric since the houses can get sunlight from both of those sides, but north and south play very different roles since houses cannot get sunlight from the north.

In Figures \ref{fig:M_33}, \ref{fig:M_34} and \ref{fig:M_43} one can see how the set $\calM_{m, n}$ looks like for several different values of $m$ and $n$. More precisely, we showed that $\aps{\calM_{3, 3}} = 10$, $\aps{\calM_{3, 4}} = 19$ and $\aps{\calM_{4, 3}} = 25$. A natural question that arises is the following.

\begin{question}
    What is the cardinality of the set $\calM_{m, n}$ (for $m, n \in \bbN$)?
\end{question}

Using the terminology from \cite{PSZ-21}, maximal configurations (for a particular $m, n \in \bbN$) with the highest building density possible are called \emph{efficient}, while those with the lowest building density possible are called \emph{inefficient}. Another interesting thing that can be observed from the previous examples is that among all the maximal configurations, the ones that have the most permutations mapped to them are some of the efficient ones. Furthermore, maximal configurations that have the least permutations mapped to them are the inefficient ones. However, there are more maximal configurations that are not efficient than the efficient ones. Therefore, it can happen that more permutations (in total) are mapped to maximal configurations that do not have the highest building density possible than to the efficient ones. For bigger tracts of land, it becomes quite computationally demanding to evaluate the function $G$ for all permutations in $S_{[m]\times [n]}$, but we can sample some permutations from $S_{[m]\times [n]}$ uniformly at random and see how the function $G$ acts on them. We did this for a $5 \times 6$ tract of land. We sampled $100\,000$ random permutations and we acted with the function $G$ on them. The results can be seen in Figure \ref{fig:randperm}. From the box plot it is clear that again much more permutations were mapped to a particular efficient maximal configuration than to the maximal configurations that are not efficient. On the other hand, as it can be seen from the histogram, there are obviously much more maximal configurations where there are $23$ or even $22$ occupied lots (than those with $24$ occupied lots) which results in building densities $23/30$ and $22/30$ being more probable than the building density $24/30$. From \cite{PSZ-21} we know that the building density of the efficient maximal configurations on $5 \times 6$ tract of land is precisely $24/30$.

\begin{figure}
	\begin{subfigure}{0.49\textwidth}\centering
		\includegraphics[width=1\textwidth]{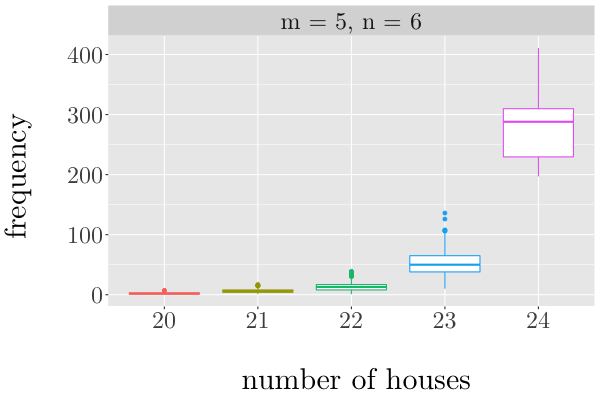}
	\end{subfigure}
	\begin{subfigure}{0.49\textwidth}\centering
		\includegraphics[width=1\textwidth]{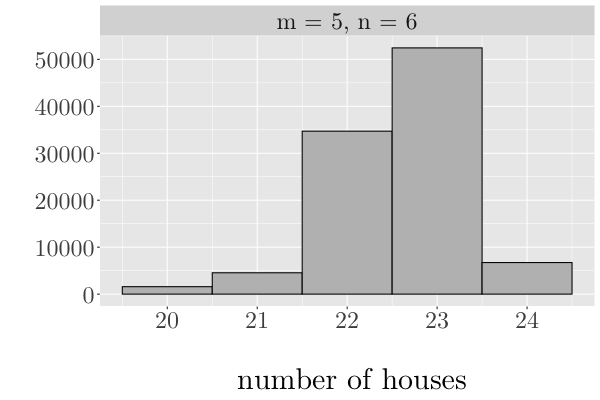}
	\end{subfigure}
	\caption{Comparison of the number of permutations (out of $100\,000$ sampled) that the function $G$ maps to a particular maximal configuration and the number of permutations that the function $G$ maps to any of the maximal configurations with a particular occupancy.}\label{fig:randperm}
\end{figure}

\begin{conjecture}\label{conj:eff-ineff+perm}
    Among all the maximal configurations, the ones that have the most permutations mapped to them (by the function $G$) are some of the efficient ones. Moreover, maximal configurations that have the fewest permutations mapped to them (by the function $G$) are some of the inefficient ones.
\end{conjecture}

Notice that the claim of Conjecture \ref{conj:eff-ineff+perm} is intuitive, but not trivial. Namely, every empty lot in the maximal configuration implies some constraints on the permutation that resulted in this particular configuration. More precisely, some specific lots had to appear in the permutation before the lot that stayed empty. The more empty lots we want to obtain in the final configuration, the more constraints on a permutation we impose and this should result in fewer permutations mapping to maximal configurations with more empty lots. On the other hand, the situation is not completely trivial since in the case when the empty lots are grouped together, they imply fewer constraints on a permutation (compared to when they are further apart) and we can expect that the more empty lots we have, the more grouped they can be. That is why, in general, there will be some efficient configurations that have fewer permutations mapped to them than to some maximal configurations with one empty lot more, but notice that this is not in a contradiction with Conjecture \ref{conj:eff-ineff+perm}.

From now on, we focus only on the random variable $X_{m, n}^s$ as this is a more realistic model for the problem that we had in mind and we also do not need to know all the maximal configurations to be able to sample this random variable.

\section{Distribution of a random variable \texorpdfstring{$X_{m, n}^s$}{Xs{m, n}}}\label{sec:dist_X^s}

For small tracts of land, we managed to obtain the exact probability distributions of variables $X_{m, n}^u$ and $X_{m, n}^s$ by exhaustive search. However, this soon becomes infeasible due to the computational complexity. Hence, we have another interesting question that one could attack analytically.

\begin{question}
    What is the distribution of random variables $X_{m, n}^u$ and $X_{m, n}^s$ (for $m, n \in \bbN$)?
\end{question}

Since we were not able to solve this problem analytically, we turned to simulations. We obtained an approximate distribution of a random variable $X_{m, n}^s$ by sampling a big number of permutations from $S_{[m]\times [n]}$, acting with function $G$ on those permutations, and then counting the number of occupied lots in the resulting maximal configurations. The results can be seen in Figure \ref{fig:histograms-dist-X}.
\begin{figure}
	\begin{subfigure}{0.49\textwidth}\centering
		\includegraphics[width=1\textwidth]{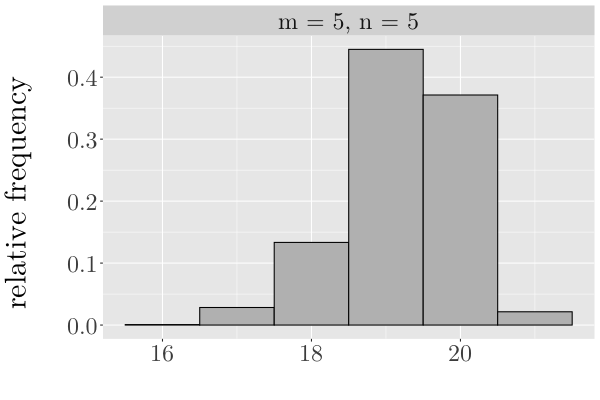}
	\end{subfigure}
	\begin{subfigure}{0.49\textwidth}\centering
		\includegraphics[width=1\textwidth]{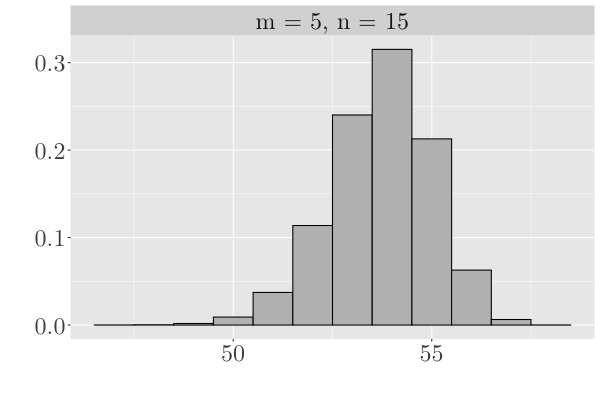}
	\end{subfigure}
	\begin{subfigure}{0.49\textwidth}\centering
		\includegraphics[width=1\textwidth]{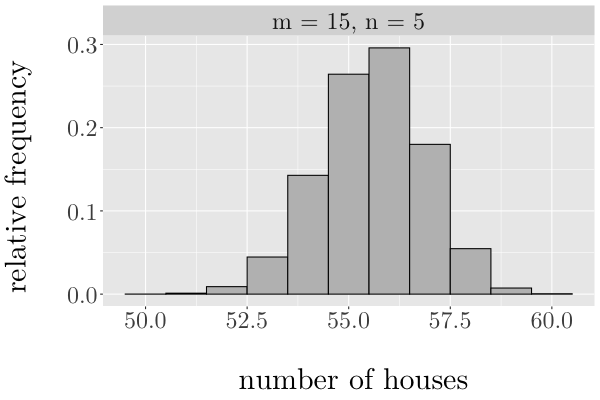}
	\end{subfigure}
	\begin{subfigure}{0.49\textwidth}\centering
		\includegraphics[width=1\textwidth]{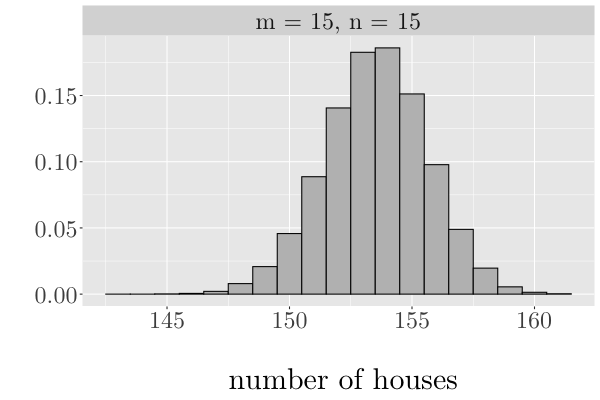}
	\end{subfigure}
	\caption{Histograms that approximate the distribution of the random variable $(mn) \cdot X_{m, n}^s$, i.e.\ the total number of houses (occupancy) in a maximal configuration.}\label{fig:histograms-dist-X}
\end{figure}
We simulated the distribution of random variables $X_{5, 5}^s$, $X_{5, 15}^s$, $X_{15, 5}^s$ and $X_{15, 15}^s$ by sampling $50\,000$ random permutations from $S_{[m]\times [n]}$ for each combination of parameters $m$ and $n$. There are several reasons for choosing this particular values of $m$ and $n$. The first reason is that for all the values of $m$ and $n$ smaller than $17$ the authors in \cite{PSZ-21} calculated precise values of the highest and the lowest building density so we already know what is the support of the random variables we simulated. The other reason is that in the next chapter we will be interested in what happens with the $\bbE[X_{m, n}^s]$ when at least one of the indices $m$ or $n$ goes to infinity. Box plots in Figure \ref{fig:box_plots-dist-X_mn} already suggest that increasing $m$ or $n$ causes expectation to drop and the size of the drop depends on whether we increase only $m$, only $n$ or both $m$ and $n$. This is exactly what we will show in the next section, with the aid of simulations.

\begin{figure}[ht]
	\centering
	\includegraphics[scale=0.58]{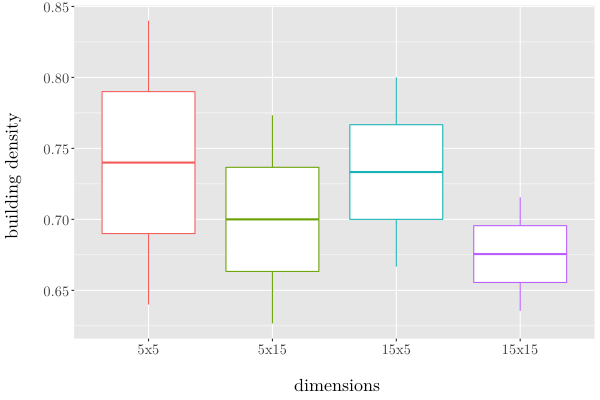}
	\caption{Comparison of distributions of random variables $X_{m, n}^s$ for different values of $m$ and $n$.} \label{fig:box_plots-dist-X_mn}
\end{figure}

\section{Packing density}\label{sec:packingDensity}

Sequential building of houses has a lot of similarities with the so-called random sequential adsorption (RSA). RSA refers to a process where particles are randomly introduced in a system and, if they do not overlap any previously adsorbed particle, they adsorb and remain fixed for the rest of the process. In our model, we assume that houses are built one-by-one. Once a house is built, we can say that it is adsorbed and that it remains fixed for the rest of the process. However, even though it will never happen that an occupied lot is chosen (hence, there are no overlaps like in a standard RSA model), it can happen that a house cannot be built on the chosen lot (i.e.\ it is not adsorbed) if that house would block some other house from the sunlight or if that house itself would be blocked from the sunlight. RSA can be carried out in computer simulations, in a mathematical analysis, or in experiments. The same holds for our model. Our main goal for the future work is to find a way how to mathematically analyze our model, but to get a better understanding of what to expect, we first carried out computer simulations of this model. The standard RSA was first studied by one-dimensional models: the attachment of pendant groups in a polymer chain by Paul Flory \cite{Flory}, and the car-parking problem by Alfr\'{e}d R\'{e}nyi \cite{Renyi} and Ewan Stafford Page \cite{Page} who studied a model analogous to the one studied by R\'{e}nyi, but on a discrete interval (so called \textit{discrete R\'{e}nyi packing problem} or \textit{unfriendly seating problem}). Other early works include those of Benjamin Widom \cite{Widom}. In two and higher dimensions many systems have been studied by computer simulations, including disks, randomly oriented squares and rectangles, aligned squares and rectangles, $k$-mers (particles occupying $k$ adjacent sites) and various other shapes.

An important result related to these models is the maximum surface coverage, called the saturation coverage, the packing fraction, or  the packing density; also known as the jamming limit in relation to parking problems. This is exactly the question that was in the center of our attention. More precisely, we were interested in values
\begin{equation}\label{eq:key_limits}
    \lim_{m \to \infty} \bbE[X_{m, n}^s] \,\,\,\, (\textnormal{for fixed } n), \quad \lim_{n \to \infty} \bbE[X_{m, n}^s] \,\,\,\, (\textnormal{for fixed } m), \quad \lim_{\substack{n \to \infty \\ m \to \infty}} \bbE[X_{m, n}^s].
\end{equation}
The precise constants in the limit are called the packing densities. There are some famous packing densities calculated explicitly and many more estimated using simulations. For the one-dimensional (continuous) car-parking problem, R\'{enyi} has shown in \cite{Renyi} that the jamming limit is equal to
\begin{equation*}
    \theta_1^c = \int_{0}^{\infty} \exp \OBL{-2 \int_0^x \frac{1 - e^{-y}}{y} dy} dx \approx 0.7475979,
\end{equation*}
the so-called R\'{e}nyi car-parking constant. For the discrete version of the R\'{e}nyi car-parking problem already Flory in \cite{Flory} observed that the jamming limit is
\begin{equation*}
    \theta_1^d = 1 - e^{-2} \approx 0.8646647,
\end{equation*}
and then Page (in \cite{Page}) proved that we even have convergence in probability and not only convergence of expectations.

\medskip

For saturation coverage of $k$-mers on one-dimensional lattice systems, see \cite{Krapivsky_et_al}. Analogous results on two-dimensional lattice systems can be found in \cite{Tarasevich_et_al} and \cite{Slutskii_et_al}. Saturation coverage of $k \times k$  squares on a two-dimensional square lattice is studied in \cite{Privman_et_al} and \cite{Brosilow_et_al}. The model which has some resemblance to our model is the one studied in \cite{Gan-Wang}, but this is still in the realm of the standard random sequential adsorption, unlike our model. For a systematic overview of all of the mentioned results and many more, see \cite{wiki-RSA}.

To get a better insight in the behaviour of $\bbE[X_{m, n}^s]$, when we let $m$ or $n$ to infinity, we ran several simulations (results of which are shown in Figures \ref{fig:jamming_limits_nx5_and_5xn} and \ref{fig:jamming_limit_nxn}). First we fixed $m = 5$ and varied $n$ from $10$ to $100$ with the step size $10$. For each combination of $m$ and $n$, we ran $2\,000$ simulations and calculated the mean building density and furthermore 5\textsuperscript{th} and 95\textsuperscript{th} percentile of the obtained building densities. Means are shown with the black dots that are connected with lines, and the area from 5\textsuperscript{th} to 95\textsuperscript{th} percentile is shaded gray. Then we switched the roles of $m$ and $n$ and repeated the same procedure. And finally, we ran the simulations where we varied both $m$ and $n$ from $10$ to $100$ with step size $10$. Based on the results of the simulations shown in Figures \ref{fig:jamming_limits_nx5_and_5xn} and \ref{fig:jamming_limit_nxn} we pose the following conjecture:

\begin{conjecture}\label{conj:E[X_mn]-decreasing}
    The double sequence $\{\bbE[X_{m, n}^s]\}_{m, n}$ is decreasing in $m$ and in $n$.
\end{conjecture}

\begin{figure}
	\begin{subfigure}{0.49\textwidth}\centering
		\includegraphics[width=1\textwidth]{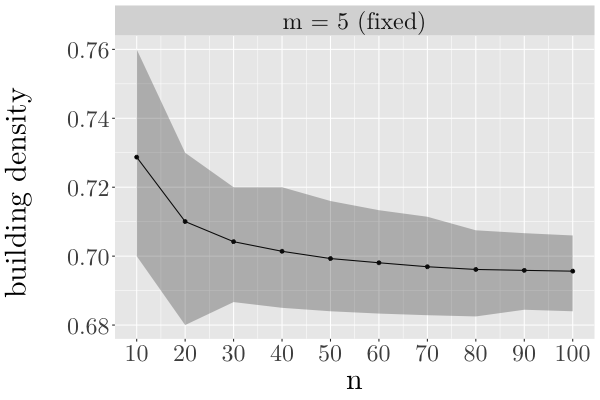}
	\end{subfigure}
	\begin{subfigure}{0.49\textwidth}\centering
		\includegraphics[width=1\textwidth]{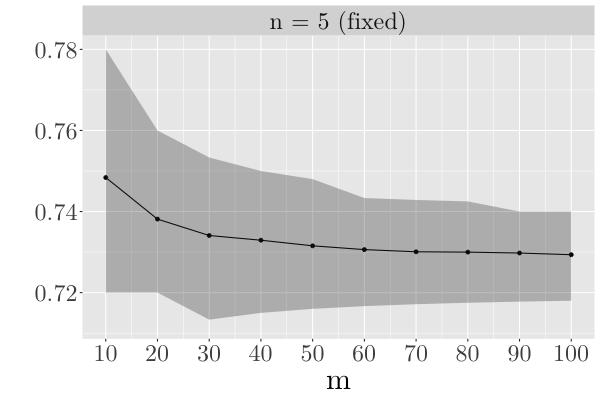}
	\end{subfigure}
	\caption{Behavior of $\bbE[X_{m, n}^s]$ when we let just one of the parameters, $m$ or $n$, to infinity.}\label{fig:jamming_limits_nx5_and_5xn}
\end{figure}

\begin{figure}
	\centering
	\includegraphics[scale=0.58]{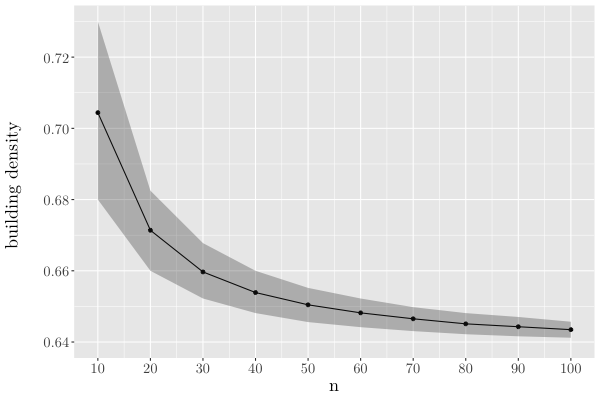}
	\caption{Behavior of $\bbE[X_{m, n}^s]$ when we let both of the parameters, $m$ and $n$, to infinity at the same rate.} \label{fig:jamming_limit_nxn}
\end{figure}

Since we know that this expectation is strictly positive, from Conjecture \ref{conj:E[X_mn]-decreasing} it would immediately follow that all the limits from \eqref{eq:key_limits} exist. We give Monte Carlo estimation of these limits in Table \ref{tab:jamming_limits}. For each combination of parameters $m$ and $n$, we ran $100$ simulations and then calculated the mean.

\begin{table}
\begin{tabular}{c|c|c|c|c|c|c|c|c}
    m $\setminus$ n & $5$ & $10$ & $50$ & $100$ & $500$ & $1\,000$ & $10\,000$ & $100\,000$ \\
    \hline
    $5$ & $0.769$ & $0.728$ & $0.699$ & $0.697$ & $0.693$ & $0.692$ & $0.692$ & $0.692$ \\
    \hline
    $10$ & $0.749$ & $0.703$ & $0.672$ & $0.668$ & $0.665$ & $0.665$ & $0.664$ & $0.664$ \\
    \hline
    $50$ & $0.730$ & $0.684$ & $0.651$ & $0.646$ & $0.643$ & $0.642$ & $0.642$ & $0.642$ \\
    \hline
    $100$ & $0.729$ & $0.682$ & $0.648$ & $0.643$ & $0.640$ & $0.640$ & $0.639$ & $0.639$ \\
    \hline
    $500$ & $0.728$ & $0.680$ & $0.646$ & $0.641$ & $0.638$ & $0.637$ & $0.637$ & $0.637$ \\
    \hline
    $1\,000$ & $0.728$ & $0.680$ & $0.645$ & $0.641$ & $0.638$ & $0.637$ & $0.637$ & $0.637$ \\
    \hline
    $10\,000$ & $0.727$ & $0.680$ & $0.645$ & $0.641$ & $0.637$ & $0.637$ & $0.636$ & $0.636$ \\
    \hline
    $100\,000$ & $0.727$ & $0.680$ & $0.645$ & $0.641$ & $0.637$ & $0.637$ & $0.636$ & 
\end{tabular}
\vspace{8px}
\caption{Packing densities}
\label{tab:jamming_limits}
\end{table}

Already for a $100\,000 \times 100\,000$ tract of land we did not have enough computational power to run the simulation. Regardless of that, the pattern is pretty clear. Values in each particular row and each particular column are decreasing and then stabilizing at some point. Even though the literature suggests that analytically obtaining the precise packing density in any of the studied cases (letting $m$ to infinity, letting $n$ to infinity or letting both $m$ and $n$ to infinity) is very involved, one can still ask the following question

\begin{question}\label{q:precise_packing_density}
	What are the exact values of the limits
	\begin{equation*}
		\lim_{m \to \infty}{\bbE[X_{m, n}^s]}, \qquad \lim_{n \to \infty}{\bbE[X_{m, n}^s]} \qquad \textnormal{and} \qquad \lim_{\substack{m \to \infty \\ n \to \infty}}{\bbE[X_{m, n}^s]}?
	\end{equation*}
\end{question}

On the other hand, there are many different techniques developed for approximating packing densities in the earlier mentioned models of random sequential adsorption so even if Question \ref{q:precise_packing_density} turns out to be too demanding, there is still a question of obtaining better approximations for packing density, using simulations designed in a better and more appropriate way or adopting some of the approaches from the RSA literature.

\section*{Acknowledgments} 

We wish to thank Juraj Bo\v{z}i\'{c} who introduced us to this problem that he came up with during his studies at Faculty of Architecture, University of Zagreb.

We additionally want to thank our colleagues Petar Baki\'{c}, Matija Ba\v{s}i\'{c} and Stipe Vidak thank to whom one particular instance of this problem ended up in the 10\textsuperscript{th} Middle European Mathematical Olympiad in V\"{o}klabruck, Austria (see \cite{MEMO}).

We also wish to thank Professor Tomislav Do\v{s}li\'{c} for fruitful and stimulating discussions.

\bibliographystyle{babamspl}
\bibliography{literature}

\begin{thebibliography}{10}
  \providecommand{\bysame}{\leavevmode\hbox to3em{\hrulefill}\thinspace}
  \providecommand{\MR}{\relax\ifhmode\unskip\space\fi MR }
  \providecommand{\MRhref}[2]{%
    \href{http://www.ams.org/mathscinet-getitem?mr=#1}{#2}
  }
  \providecommand{\href}[2]{#2}
  \providebibliographyfont{name}{}%
  \providebibliographyfont{lastname}{}%
  \providebibliographyfont{title}{\emph}%
  \providebibliographyfont{jtitle}{\btxtitlefont}%
  \providebibliographyfont{etal}{}%
  \providebibliographyfont{journal}{}%
  \providebibliographyfont{volume}{\textbf}%
  \providebibliographyfont{ISBN}{\MakeUppercase}%
  \providebibliographyfont{ISSN}{\MakeUppercase}%
  \providebibliographyfont{url}{\url}%
  \providebibliographyfont{numeral}{}%
  \providecommand\btxprintamslanguage[1]{\ (#1)}
  \expandafter\btxselectlanguage\expandafter {\btxfallbacklanguage}

\expandafter\btxselectlanguage\expandafter {\btxfallbacklanguage}
\bibitem {Brosilow_et_al}
\btxnamefont {B.\btxfnamespacelong J. \btxlastnamefont {Brosilow}},
  \btxnamefont {R.\btxfnamespacelong M. \btxlastnamefont {Ziff}}\btxandcomma {}
  \btxandlong {} \btxnamefont {R.\btxfnamespacelong D. \btxlastnamefont
  {Vigil}}, \btxjtitlefont {\btxifchangecase {Random sequential adsorption of
  parallel squares}{Random sequential adsorption of parallel squares}},
  \btxjournalfont {Phys. Rev. A.} \btxvolumefont {43} (1991), 631--638.

\bibitem {Flory}
\btxnamefont {P.\btxfnamespacelong J. \btxlastnamefont {Flory}}, \btxjtitlefont
  {\btxifchangecase {Intramolecular reaction between neighboring substituents
  of vinyl polymers}{Intramolecular reaction between neighboring substituents
  of vinyl polymers}}, \btxjournalfont {Journal of the American Chemical
  Society} \btxvolumefont {61} (1939), \btxnumbershort {.}~6, 1518--1521.

\bibitem {Gan-Wang}
\btxnamefont {C.\btxfnamespacelong K. \btxlastnamefont {Gan}} \btxandlong {}
  \btxnamefont {J.\btxfnamespacelong S. \btxlastnamefont {Wang}},
  \btxjtitlefont {\btxifchangecase {Extended series expansions for random
  sequential adsorption}{Extended series expansions for random sequential
  adsorption}}, \btxjournalfont {J. Chem. Phys.} \btxvolumefont {108} (1998),
  3010--3012.

\bibitem {Krapivsky_et_al}
\btxnamefont {P.\btxfnamespacelong L. \btxlastnamefont {Krapivsky}},
  \btxnamefont {S.~\btxlastnamefont {Redner}}\btxandcomma {} \btxandlong {}
  \btxnamefont {E.~\btxlastnamefont {Ben-Naim}}, \btxtitlefont
  {\btxifchangecase {A kinetic view of statistical physics}{A kinetic view of
  statistical physics}}, \btxpublisherfont {Cambridge University Press,
  Cambridge}, 2010, \mbox{\btxISBN~\btxISBNfont {978-0-521-85103-9}}.

\bibitem {MEMO}
\btxnamefont {Tenth {M}iddle {E}uropean~{M}athematical \btxlastnamefont
  {{O}lympiad}}, \btxjtitlefont {\btxifchangecase {{C}ontest problems with
  solutions}{{C}ontest Problems with Solutions}}, \btxjournalfont
  {V\"{o}cklabruck, Austria} (2016), {\latintext
  \btxurlfont{https://www.math.aau.at/MEMO2016/wp-content/uploads/2021/04/MEMO2016_Solutions-8.pdf}}.

\bibitem {Page}
\btxnamefont {E.\btxfnamespacelong S. \btxlastnamefont {Page}}, \btxjtitlefont
  {\btxifchangecase {The distribution of vacancies on a line}{The distribution
  of vacancies on a line}}, \btxjournalfont {J. Roy. Statist. Soc. Ser. B}
  \btxvolumefont {21} (1959), 364--374.

\bibitem {Privman_et_al}
\btxnamefont {V.~\btxlastnamefont {Privman}}, \btxnamefont
  {J.\btxfnamespacelong S. \btxlastnamefont {Wang}}\btxandcomma {} \btxandlong
  {} \btxnamefont {P.~\btxlastnamefont {Nielaba}}, \btxjtitlefont
  {\btxifchangecase {Continuum limit in random sequential adsorption}{Continuum
  limit in random sequential adsorption}}, \btxjournalfont {Phys. Rev. B.}
  \btxvolumefont {43} (1991), 3366--3372.

\bibitem {PSZ-21}
\btxnamefont {M.~\btxlastnamefont {Puljiz}}, \btxnamefont {S.~\btxlastnamefont
  {\v{S}ebek}}\btxandcomma {} \btxandlong {} \btxnamefont {J.~\btxlastnamefont
  {\v{Z}ubrini\'{c}}}, \btxjtitlefont {\btxifchangecase {Combinatorial
  settlement planing}{Combinatorial settlement planing}},  (2021), {\latintext
  \btxurlfont{https://arxiv.org/abs/2107.07555}}.

\bibitem {Renyi}
\btxnamefont {A.~\btxlastnamefont {R\'{e}nyi}}, \btxjtitlefont
  {\btxifchangecase {On a one-dimensional problem concerning random space
  filling}{On a one-dimensional problem concerning random space filling}},
  \btxjournalfont {Magyar Tud. Akad. Mat. Kutat\'{o} Int. K\"{o}zl.}
  \btxvolumefont {3} (1958), \btxnumbershort {.}~1-2, 109--127.

\bibitem {Slutskii_et_al}
\btxnamefont {M.\btxfnamespacelong G. \btxlastnamefont {Slutskii}},
  \btxnamefont {L.\btxfnamespacelong Y. \btxlastnamefont {Barash}}\btxandcomma
  {} \btxandlong {} \btxnamefont {Y.\btxfnamespacelong Y. \btxlastnamefont
  {Tarasevich}}, \btxjtitlefont {\btxifchangecase {Percolation and jamming of
  random sequential adsorption samples of large linear $k$-mers on a square
  lattice}{Percolation and jamming of random sequential adsorption samples of
  large linear $k$-mers on a square lattice}}, \btxjournalfont {Phys. Rev. E.}
  \btxvolumefont {98} (2018).

\bibitem {Tarasevich_et_al}
\btxnamefont {Y.\btxfnamespacelong Y. \btxlastnamefont {Tarasevich}},
  \btxnamefont {V.\btxfnamespacelong V. \btxlastnamefont {Laptev}}\btxandcomma
  {} \btxandlong {} \btxnamefont {N.\btxfnamespacelong V. \btxlastnamefont
  {Vygornitskii}}, \btxjtitlefont {\btxifchangecase {Impact of defects on
  percolation in random sequential adsorption of linear $k$-mers on square
  lattices}{Impact of defects on percolation in random sequential adsorption of
  linear $k$-mers on square lattices}}, \btxjournalfont {Phys. Rev. E.}
  \btxvolumefont {91} (2015).

\bibitem {Widom}
\btxnamefont {B.~\btxlastnamefont {Widom}}, \btxjtitlefont {\btxifchangecase
  {Random sequential addition of hard spheres to a volume}{Random Sequential
  Addition of Hard Spheres to a Volume}}, \btxjournalfont {The Journal of
  Chemical Physics} \btxvolumefont {44} (1966), 3888--3894.

\bibitem {wiki-RSA}
\btxnamefont {\btxlastnamefont {Wikipedia}}, \btxjtitlefont {\btxifchangecase
  {Random sequential adsorption}{Random sequential adsorption}}, {\latintext
  \btxurlfont{https://en.wikipedia.org/wiki/Random_sequential_adsorption}}.

\end{thebibliography}

\end{document}